\documentclass[10pt,final,leqno,onefignum,onetabnum]{siamltex}%{article}%
\usepackage{amssymb,amsfonts,amsmath,bm}
\usepackage{hyperref}
\usepackage{psfrag,graphicx,color}
\usepackage{mathrsfs}
\usepackage{graphics}
\usepackage{subfigure}
\usepackage{enumerate,enumitem}
\usepackage{pgfplots}
\usetikzlibrary{arrows.meta}
\usepackage{MnSymbol}
\usepackage{cite}
\usepackage{diagbox}%
\usepackage[notref,notcite]{showkeys} % To see crossreferences.
\newtheorem{remark}{Remark}[section]

\definecolor{darkred}{rgb}{0.85,0,0}

\definecolor{green}{rgb}{0,0.7,0}

\def\d{\textrm{d}}
\def\R{\mathbb{R}}

\def\leq{\leqslant}

\def\l{\left}
\def\r{\right}

\def\d{\mathrm{d}}

\def\epsilon{\varepsilon}
\def\II{(\Omega)}

\def\fy{\varphi}

\allowdisplaybreaks[1]
\begin{document}

\title{Arbitrarily High-order Maximum Bound Preserving Schemes with Cut-off Postprocessing for Allen-Cahn Equations
\thanks{The work of J.\ Yang is supported
 by National Natural Science Foundation of China (NSFC) Grant No. 11871264, Natural Science Foundation of Guangdong Province (2018A0303130123),
 and NSFC/Hong Kong RRC Joint Research Scheme (NFSC/RGC 11961160718), and
the research of Z.\ Yuan and Z.\ Zhou is partially supported by Hong Kong RGC grant (No. 15304420).}}

\author{Jiang Yang\thanks{Department of Mathematics \& SUSTech International Center for Mathematics, 
Southern University of Science and Technology, Shenzhen 518055, China.
(\texttt{yangj7sustech.edu.cn})}
\and Zhaoming Yuan\thanks{Department of Applied Mathematics, The Hong Kong Polytechnic University, Kowloon, Hong Kong.
%Department of Applied Physics and Applied Mathematics, Columbia University, 500 W. 120th Street, New York, NY 10027, USA
(\texttt{zhaoming.yuan@connect.polyu.hk})}
\and Zhi Zhou\thanks{Department of Applied Mathematics, The Hong Kong Polytechnic University, Kowloon, Hong Kong.
%Department of Applied Physics and Applied Mathematics, Columbia University, 500 W. 120th Street, New York, NY 10027, USA
(\texttt{zhizhou@polyu.edu.hk, zhizhou0125@gmail.com})}}
\date{\today}

 \maketitle
 
\begin{abstract}
We develop and analyze a class of maximum bound preserving schemes for approximately solving Allen--Cahn equations.
We apply a $k$th-order single-step scheme in time (where the nonlinear term is linearized by multi-step extrapolation),
and  a lumped mass finite element method in space with piecewise
$r$th-order polynomials and Gauss--Lobatto quadrature. At each time level, a cut-off post-processing
is proposed to eliminate extra values violating the maximum bound principle at the finite element nodal points.
As a result, the numerical solution satisfies the maximum bound principle (at all nodal points),
and the optimal error bound $O(\tau^k+h^{r+1})$ is theoretically proved for a certain class of schemes.
These time stepping schemes under consideration includes algebraically stable collocation-type methods,
which could be arbitrarily high-order in both space and time.
%, provided that the solution is sufficiently smooth
%and the single-step scheme satisfies certain assumptions, e.g., the $L$-stability.
%Next, we discuss the algebraically stable collocation-type methods, which belongs to a specified class of single step methods,
%and establish the error bound $O(\tau^k+h^{r+1})$ without the $L$-stability.
Moreover, combining the cut-off strategy with the scalar auxiliary value (SAV) technique,
we develop a class of energy-stable and maximum bound preserving schemes, which is arbitrarily high-order in time.
Numerical results are provided to illustrate the accuracy of the proposed method.\\


\end{abstract}

\section{Introduction}\label{Se:intro}
The aim of this paper is to design and analyze a high-order maximum bound preserving (MBP) scheme for solving the Allen--Cahn equation:
\begin{equation}\label{eqn:AC}
    \begin{cases}
    u_t = \Delta u + f(u) &
        \mbox{in~~} \Omega\times(0,T),\\
    u(x,t=0) = u_0(x)   &
        \mbox{in~~} \Omega\times\{0\},\\
    \partial_{\mathbf{n}}u = 0  &
        \mbox{on~~} \partial\Omega\times(0,T)\\
  \end{cases}
\end{equation}
where $\Omega$ is a smooth domain in $\mathbb{R}^d$ with the boundary $\partial\Omega$.
$f(u)=-F'(u)$ with a double-well potential $F$ that has two wells at $\pm\alpha$, for {some known parameter $\alpha>0$}. %; see Examples \ref{AC:1} and \ref{AC:2}
For two popular choices of potentials. It is well-known that the Allen--Cahn equation \eqref{eqn:AC} has the maximum bound principle \cite{DuJuLiQiao:2020}:
\begin{align}\label{eqn:max-AC}
|u_0(x)|\le \alpha\quad \Longrightarrow\quad |u(x,t)| \le \alpha\qquad \text{for all} ~~(x,t)\in \Omega\times(0,T].
\end{align}
As a typical $L^2$ gradient flow associating with the following free energy
$$E(u)=\int_\Omega\frac{1}{2}|\nabla u|+F(u)\d x,$$
 the nonlinear energy dissipation law holds in the sense
\begin{equation}\label{enlaw}\frac{\d}{\d t}E(u)=-\int_\Omega|u_t|^2\d x\le 0.\end{equation}

The Allen--Cahn equation was originally introduced by Allen and Cahn in \cite{AC} to describe the motion of anti-phase boundaries in crystalline solids. In the context, $u$ represents the concentration of one of the two metallic components of the alloy and
the parameter $\epsilon$ involved in the nonlinear term  represents the interfacial width, which is small compared to the
characteristic length of the laboratory scale. Recent decades, the Allen--Cahn equation has become one of basic phase-field equations, which has been widely applied to many complicated moving interface problems in materials science and fluid dynamics through a phase-field approach coupled with other models \cite{anderson1998diffuse,chen2002phase,yue2004diffuse}.

\subsection{Review on existing studies}
The development and analysis of MBP method have been intensively studied in existing references.
It was proved in \cite{tangyang1,ShenTangYang:2016} that the stabilized semi-implicit Euler time-stepping scheme,
with central difference method in space, preserves the maximum principle unconditionally
if the stabilizer satisfies certain restrictions. In \cite{DuJuLiQiao:2019},
a stabilized exponential time differencing scheme was proposed for solving the (nonlocal) Allen--Cahn equation,
and the scheme was proved to be unconditionally MBP. See also \cite{DuJuLiQiao:2020}
for the generalization to a class of semilinear parabolic equations.
%To the best of my knowledge, this is the highest-order linearly implicit MMP method for the Allen--Cahn equation without any stepsize conditions.
The second-order backward differentiation formula (with nonuniform meshes) was applied to develop an MBP scheme
in \cite{LiaoTangZhou:2020} under the usual CFL condition $\tau=O(h^2)$.

High-order strong stability preserving (SSP) time-stepping methods are widely used
in the development of MBP scheme for both parabolic equations and hyperbolic equations (see e.g.,  
\cite{gottlieb1998total,Liu-Yu-2014,GottliebShuTadmor:2001,gottlieb2011strong,Liu-Osher-1996,Qiu-Shu-2005,Xu-2014,Zhang-Shu-2010}).
Recently, an SSP integrating factor Runge--Kutta method of up to order four was proposed and
analyzed in \cite{IsherwoodGrantGottlieb:2018} for semilinear hyperbolic and parabolic equations.
For semilinear hyperbolic and parabolic equations with strong stability (possibly in the maximum norm), the method
can preserve this property and can avoid the standard parabolic CFL condition $\tau=O(h^2)$, only requiring
the stepsize $\tau$ to be smaller than some constant depending on the nonlinear source term, also referring to \cite{ju2020maximum}.
A nonlinear constraint limiter was introduced in \cite{Vegt-Xia-Xu-2019} for implicit time-stepping
schemes without requiring CFL conditions, which can preserve maximum principle at the
discrete level with arbitrarily high-order methods by solving a nonlinearly implicit system.

Very recently, a new class of high-order MBP methods was proposed in \cite{LiYangZhou:sisc}.
The method consists of a $k$th-order multistep exponential integrator
in time, and a lumped mass finite element method in space with piecewise $r$th-order
polynomials. At every time level, the extra values exceeding the maximum bound
are eliminated at the finite element nodal points by a cut-off operation.
Then the numerical solution at all nodal points
satisfies the MBP, and an error bound of $O(\tau^k + h^r)$ was proved. However, numerical results in \cite[Table 4.1]{LiYangZhou:sisc}
indicates that the error bound is not sharp in space, and how to improve the estimate it is still open.
%As far as we know, this is the first work presenting arbitrarily high-order unconditionally MBP schemes with provable error estimates.
Besides, the aforementioned scheme requires to evaluate some actions of exponential functions of diffusion operators,
which might be relatively expensive compared with solving poisson problems, and the generalization to other time stepping schemes is a nontrivial task.
Finally, the proposed scheme (with relatively coarse step sizes) might produce a numerical solution with obviously
increasing and oscillating energy. These motivate our current project.

\subsection{Our contributions and the organization of the paper}
The first contribution of the current paper is to develop and analyze
a class of MBP schemes, which could be arbitrarily high-order in both space and time,
for approximately solving the Allen--Cahn equation \eqref{eqn:AC}.
In time, we apply a single-step method, which is (strictly) accurate of order $k$,
and apply multistep extrapolation to linearize the nonlinear term.
In space, we apply the lumped mass FEM with piecewise $r$th-order
polynomials and Gauss--Lobatto quadrature, as in \cite{LiYangZhou:sisc}.
At each time level, we apply a cut-off operation to remove the extra value exceeding the maximum bound
at the nodal points. We estabilish the error estimate of order $O(\tau^k + h^{r+1})$,
which fills the gap between the numerical results in \cite[Theorem 3.2]{LiYangZhou:sisc} showing optimal convergence rate $O(h^{r+1})$
and the theoretical result in \cite[Table 4.1]{LiYangZhou:sisc} providing only a suboptimal error estimate of order $O(h^{r})$.
The improvement follows from a careful examination of quadrature errors
(see Remark \ref{rem:optimal-LiYangZhou} and \cite[eq. (2.6) and (3.22)]{LiYangZhou:sisc}).
To the best of our knowledge, this is the first work deriving optimal error estimates
of  arbitrarily high-order MBP schemes for the Allen--Cahn equation  \eqref{eqn:AC}.

Nevertheless, the optimal estimate of the fully discrete scheme (with the cut-off postprocessing) requires the L-stability of the time stepping scheme,
which excludes some popular and practical singe step method, e.g. Gauss--Legendre method belonging to algebraically stable collocation Runge--Kutta method.
Therefore, we revisit this class of time stepping methods and prove the same error estimate
by using the energy argument without using the L-stability. This is the second contribution of the paper.

In case of relative coarse step sizes, the proposed time stepping scheme (with the cut-off operation at each time level)
might result in oscillating and increasing energy (see e.g. Figure \ref{fig:energy} (middle)),
which violates the energy dissipation law \eqref{enlaw} of Allen--Cahn equation  \eqref{eqn:AC}. This motivates us to
develop a class of energy-stable and MBP schemes,
by combining the cut-off strategy with the scalar auxiliary value (SAV) method \cite{shen2019new}.
The scheme is second order in space but could be arbitrarily high-order in time. As far as we know, this is the first scheme that is unconditionally
 energy-dissipative, maximum bound preserving, and
arbitrarily high-order in time scheme with a provable error bound.
In fact, our numerical results show that the use of SAV regularizes the numerical solution, stabilizes the energy,
and significantly reduces the cut-off values at each time level (see e.g. Figure \ref{fig:energy}).

The rest of the paper is organized as follows. In section \ref{sec:single}, we consider the single step methods (in a general framework)
with cut-off postprocessing and multistep extrapolation. Error estimate for both semidiscrete and fully discrete scheme are provided,
where the optimal error estimate of the fully discrete scheme requires the L-stability of the time stepping scheme.
In section \ref{sec:collocation}, we analyze the algebraically stable collocation scheme and show the same error estimate without using the L-stability.
In section \ref{sec:sav}, combining the cut-off strategy with the scalar auxiliary value (SAV) method,
we develop a class of energy-stable and maximum bound preserving schemes, which is arbitrarily high-order in time.
In section \ref{sec:numerics}, we present numerical results to illustrate the
accuracy and effectiveness of the method in solving the Allen--Cahn equation.
Throughout, the notation $C$, with or without subscripts, denotes a generic constant, which may differ at different occurrences,
but it is always independent of the mesh size $h$ and the time step size $\tau$.

\section{Cut-off single-step methods with multi-step extrapolation}\label{sec:single}

%\subsection{Runge-Kutta method}
In this section, we shall develop and analyze a class of MBP scheme for the Allen--Cahn equation \eqref{eqn:AC}.
Optimal error estimate will also be provided,  which fill the gap in the preceding work \cite{LiYangZhou:sisc}.
Besides, the argument presented in this section also builds the foundation of developing MBP
scheme which also satisfies energy dissipation property (in section \ref{sec:sav}).

\subsection{Temporal semi-discrete scheme}
To begin with, we consider the time discretization for the Allen--Cahn equation \eqref{eqn:AC}.
To this end, we split the interval $(0,T)$ into $N$ subintervals with the uniform mesh size $\tau=T/N$,
and set $t_n = n\tau$, $n=0,1,\ldots,N$.
On the time interval $[t_{n-1},t_n]$, we approximate the nonlinear term $f(u(s))$ by the extrapolation polynomial
$$
 \sum_{j=1}^{k} L_j(s) f(u^{n-j}),\quad \text{with known}~ u^{n-k},\ldots,u^{n-1}.
$$
where $L_j(s)$ is the Lagrange basis polynomials of degree $k-1$ in time, satisfying
$$
L_j(t_{n-i}) = \delta_{ij}, \quad i,j=1,\dots,k.
$$
Thus, on $[t_{n-1},t_n]$, the linearization of \eqref{eqn:AC} states as
$$\tilde{u}_t=\Delta\tilde{u}+\sum_{j=1}^{k} L_j(s) f(u^{n-j}).$$
Following Duhamel's principle yields
$$\tilde{u}(t_{n})=e^{\tau\Delta}u(t_{n-1})+\int_0^\tau e^{(\tau-s)\Delta}\sum_{j=1}^{k} L_j(t_{n-1}+s) f(u^{n-j})\d s.$$
Then a framework of a single step scheme of approximating $\tilde{u}(t_{n})$ reads:
\begin{equation}\label{eqn:semi}
\tilde u^{n} = \sigma(-\tau\Delta) u^{n-1} + \tau \sum_{i=1}^m p_i(-\tau \Delta)
\Big(\sum_{j=1}^{k} L_j (t_{ni}) f(u^{n-j})\Big),\quad \text{for all}~n\ge k,
\end{equation}
with $t_{ni}=t_{n-1} + c_i\tau$.
Here,
$\sigma(\lambda)$ and $\{p_i(\lambda)\}^m_{i=1}$ are rational functions and $c_i$ are distinct real numbers in $[0,1]$.
For simplicity, we assume that the scheme \eqref{eqn:semi} satisfies the following assumptions.

\begin{itemize}
\item[{\bf (P1)}] $|\sigma(\lambda)|< 1$ and $|p_i(\lambda)|\le c$, for all $i=1,\ldots,m$,
uniformly in $\tau$ and $\lambda > 0$. Besides, the numerator of $p_i(\lambda)$
is of lower degree than its denominator.
\item[{\bf (P2)}] The time stepping scheme \eqref{eqn:semi}  is accurate of order $k$ in sense that
$$ \sigma(\lambda) = e^{-\lambda} + O(\lambda^{k+1}),\quad \text{as}~\lambda\rightarrow0. $$
and, for $0\le j\le k$
$$ \sum_{i=1}^m c_i^j p_i(\lambda) - \frac{j!}{(-\lambda)^{j+1}}\Big(e^{-\lambda}
- \sum_{\ell=0}^j \frac{(-\lambda)^\ell}{\ell!}\Big) = O(\lambda^{k-j}),\quad \text{as}~\lambda\rightarrow0. $$
\item[{\bf (P3)}] The time discretization scheme \eqref{eqn:semi} is strictly accurate of order $q$ in sense that
$$ \sum_{i=1}^m c_i^j p_i(\lambda) - \frac{j!}{(-\lambda)^{j+1}}\Big(\sigma(\lambda)
- \sum_{\ell=0}^j \frac{(-\lambda)^\ell}{\ell!}\Big) =0,\quad\text{for all}~ 0\le j\le q-1. $$
\end{itemize}

\begin{remark}
In practice,
it is convenient to choose $p_i(\lambda)$s that share the same denominator of $\sigma(\lambda)$, for instance:
$$\sigma(\lambda) = \frac{a_0(\lambda)}{g(\lambda)},\quad \text{and} \quad p_i(\lambda) = \frac{a_i(\lambda)}{g(\lambda)},\quad \text{for}~i=1,2,\ldots,m,$$
where $a_i(\lambda)$ and $g(\lambda)$ are polynomials.
Then the time stepping scheme \eqref{eqn:semi} could be written as
\begin{equation*}
g(-\tau\Delta)\tilde u^{n} = a_0(-\tau\Delta) u^{n-1} + \tau \sum_{i=1}^m a_i(-\tau \Delta)
\Big(\sum_{j=1}^{k} L_j (t_{ni}) f(u^{n-j})\Big),\quad \text{for all}~n\ge k.
\end{equation*}
See e.g. \cite[pp. 131]{Thomee:2006} for the construction of such rational functions satisfying the Assumptions (P1)-(P3).
\end{remark}

Unfortunately, the time stepping scheme \eqref{eqn:semi} does not satisfy the maximum bound principle. Therefore,
at each time step, we apply the cut-off operation: for $n\ge k$, we find $u^n$ such that
\begin{align}
\hat u^{n} &= \sigma(-\tau\Delta) u^{n-1} + \tau \sum_{i=1}^m p_i(-\tau \Delta)
\Big(\sum_{j=1}^{k} L_j (t_{ni}) f(u^{n-j})\Big),\label{eqn:semi-1}\\
u^{n} &=  \min(\max(\hat u^{n},-\alpha),\alpha),\label{eqn:semi-cut}
\end{align}
where $\alpha$ is the maximum bound given in \eqref{eqn:max-AC}.
%Due to the cut-off operation \eqref{eqn:semi-cut}, the semi-discrete solution
%obtained from \eqref{eqn:semi-1}-\eqref{eqn:semi-cut} satisfies
%$$
%|u^{n}(x)|\le \alpha \quad \forall\, x\in \Omega.
%$$
The accuracy of this cut-off semi-discrete method is guaranteed by the next theorem.

\begin{theorem}\label{thm:semi}
Suppose that the Assumptions (P1) and (P2) are fulfilled, and (P3) holds for $q=k$.
Let $u(t)$ be the solution to the Allen--Cahn equation, and $u^n$ be the solution to the time stepping scheme
\eqref{eqn:semi-1}-\eqref{eqn:semi-cut}.
Assume that $|u_0|\le \alpha$ and the maximum principle \eqref{eqn:max-AC} holds,
and assume that the starting values $u^j$, $j=0,\dots,k-1$, are given and
$$
|u^j|\le \alpha, \quad \text{for all}~~j=0,\dots,k-1.
$$
Then the semi-discrete solution given by  \eqref{eqn:semi-1}-\eqref{eqn:semi-cut} satisfies for all $n\ge k$
\begin{align*}%\label{eqn:semi-mbp}
|u^n|\le \alpha,
\end{align*}
and
\begin{equation*}%\label{eqn:semi-error}
\|u^n - u(t_n)\|\leq C\tau^k + C\sum_{j=0}^{k-1}\|u^j - u(t_j)\|,
\end{equation*}
provided that  $f$ is locally Lipschitz continuous, $\Delta u\in C^{k}([0,T]; L^2\II)$, $ u\in C^{k+1}([0,T];L^2\II)$ and $f(u) \in C^k([0,T];L^2\II)$.
\end{theorem}

\begin{proof}
Due to the cut-off operation \eqref{eqn:semi-cut}, the discrete
maximum bound principle follows immediately.
Then it suffices to show the error estimate.

Let $e ^n=u^n-u(t_n)$ and $\hat e ^n=\hat u^n-u(t_n)$. Since the exact solution satisfies the maximum bound \eqref{eqn:max-AC}, we have
$$  \|  e ^n  \|_{L^2\II} \le  \|  \hat e ^n  \|_{L^2\II}. $$
Then it is easy to note that
$$  \hat e^{n} = \sigma (-\tau\Delta) e ^{n-1} + \fy^n,\quad n \ge k. $$
where $\fy^n$ can be written as
\begin{align*}
  \fy^n &= -u(t_{n}) + \sigma(-\tau\Delta) u(t_{n-1}) + \tau \sum_{i=1}^m p_i(-\tau \Delta)
  \Big(\sum_{j=1}^{k} L_j (t_{ni}) f(u^{n-j})\Big) \\
  &= \tau \sum_{i=1}^m p_i(-\tau \Delta)\Big(  \sum_{j=1}^{k} L_j (t_{ni}) f(u^{n-j}) - f(t_{ni}))\Big)\\
  &\quad + \Big(-u(t_{n}) + \sigma(-\tau\Delta) u(t_{n-1}) + \tau \sum_{i=1}^m p_i(-\tau \Delta) (\partial_t u - \Delta u) (t_{ni})\Big)\\
  &=: I + II.
\end{align*}
Then the bound of $I$ follows  from the approximation property of Lagrange interpolation, the maximum bound of $u^{n-j}$ and $u(t_{n-j})$, $j=1,\ldots,k$,
the locally Lipschitz continuity of $f$, and the Assumption (P1):
\begin{align*}
   \|  I  \|_{L^2\II} &\le  \tau \sum_{i=1}^m \| p_i(-\tau \Delta)\|_{L^2\II\rightarrow L^2\II} \Big\|\sum_{j=1}^k L_j(t_{ni}) f(u(t_{n-j}))- f(u(t_{n-1} + c_i\tau))\Big\|_{L^2\II}\\
   &\qquad + \tau \sum_{i=1}^m \| p_i(-\tau \Delta)\|_{L^2\II\rightarrow L^2\II} \sum_{j=1}^k |L_j(t_{ni})|\,\| f(u^{n-j})- f(u(t_{n-j} ))\|_{L^2\II}\\
   &\le C\tau^{k+1} \|  f(u) \|_{C^{k}([t_{n-k},t_n];L^2\II)} + C \tau \sum_{j=1}^k \|  e^{n-j}  \|_{L^2\II}.
\end{align*}
Now we term to the second term $II$, which can be rewritten by Taylor's expansion at $t_{n-1}$
\begin{align*}
 II &= -\sum_{j=0}^k\frac{\tau^j}{j!} u^{(j)} (t_{n-1}) + \sigma(-\tau \Delta)u(t_{n-1})\\
  &\quad + \tau \sum_{i=1}^m p_i(-\tau \Delta) \sum_{j=0}^{k-1}
  \frac{(c_i\tau)^j}{j!}
  (u^{(j+1)} - \Delta u^{(j)})
  (t_{n-1}) + R_1 + R_2.
\end{align*}
where the remainders $R_1$ and $R_2$ are
\begin{align*}
 R_1 &= \int_{t_{n-1}}^{t_n} \frac{(t_n-s)^k}{k!} u^{(k+1)}(s) \,\d s\quad \text{and}\quad \\
  R_2 &=   \tau \sum_{i=1}^m p_i(-\tau \Delta) \int_{t_{n-1}}^{t_{n-1}+c_i\tau} \frac{(t_{n-1}+c_i\tau-s)^{k-1}}{(k-1)!} (u^{(k+1)}-\Delta u^{(k)})(s) \,\d s
\end{align*}
respectively. Hereafter, we use $u^{(j)}$ to denote the $j$th derivative in time. Then Assumption (P1) implies
\begin{align*}
 \|R_1+R_2\|_{L^2\II} &\le C \tau^{k+1} \Big(\|  u \|_{C^{k+1} ([t_{n-1},t_n];L^2\II)} + \|  \Delta u \|_{C^{k} ([t_{n-1},t_n];L^2\II)} \Big).
\end{align*}
Now we revisit the three leading terms of $II$. Note that
\begin{align*}
& -\sum_{j=0}^k\frac{\tau^j}{j!} u^{(j)} (t_{n-1}) + \sigma(-\tau \Delta)u(t_{n-1})
  + \tau \sum_{i=1}^m p_i(-\tau \Delta) \sum_{j=0}^{k-1} \frac{(c_i\tau)^j}{j!}
  (u^{(j+1)} - \Delta u^{(j)})(t_{n-1})\\
 =& \Big(-I +  \sigma(-\tau \Delta) - \tau \sum_{i=1}^m p_i(-\tau \Delta) \Delta \Big)u(t_{n-1}) \\
 & +   \sum_{j=1}^{k-1} \frac{\tau^j}{j!} \Big(- I  +  j \sum_{i=1}^m c_i^{j-1} p_i(-\tau \Delta) - \tau\sum_{i=1}^m c_i^{j} p_i(-\tau \Delta) \Delta\Big)u^{(j)}(t_{n-1})   \\
 & +  \frac{\tau^k}{k!} \Big(-I +  k\sum_{i=1}^m c_i^{k-1}p_i(-\tau\Delta)\Big )u^{(k)}(t_{n-1}) =\sum_{\ell=1}^3 II_\ell.
\end{align*}
Since the time stepping scheme is strictly accurate of order $q=k$ (by Assumption (P3)),
we have $ II_1=  II_2=0$. Meanwhile,
we apply Assumption (P3) again to arrive at for $\lambda>0$
\begin{align*}
  -1 + k \sum_{i=1}^m c_i^{k-1} p_i(\lambda) = \lambda \frac{k!}{(-\lambda)^{k+1}} \Big(\sigma(\lambda)-\sum_{\ell=0}^k\frac{(-\lambda)^\ell}{\ell!}\Big) =: \lambda \gamma(\lambda).
\end{align*}
Note that $|\gamma(\lambda)|=O(1)$ for $\lambda\rightarrow0$ (by Assumption (P2))
and $|\gamma(\lambda)|\rightarrow 0$ for  $\lambda\rightarrow +\infty$. Hence $|\gamma(\lambda)|$ is bounded
uniformly in $[0,\infty)$. Then we arrive at
\begin{align*}
\|  II_3 \|_{L^2\II} \le C \tau^{k+1} \|  \Delta u^{(k)}(t_{n-1}) \| \le  C \tau^{k+1} \|  \Delta u \|_{C^{k}([t_{n-1},t_n];L^2\II)}.
\end{align*}
In conclusion, we obtain the following estimate
\begin{align*}
\|  e ^n \|_{L^2\II} \le \|  \sigma(-\tau \Delta)e ^{n-1} \|_{L^2\II}  +  C \tau^{k+1}  + C \tau \sum_{j=1}^k \|  e^{n-j}  \|_{L^2\II}.
\end{align*}
Then the assumption (P1) leads to
\begin{align*}
\|  e ^n \|_{L^2\II} \le \|  e_h^{n-1} \|_{L^2\II}  +  C \tau^{k+1}  + C \tau \sum_{j=1}^k \|  e^{n-j}  \|_{L^2\II}.
\end{align*}
%and hence we rearrange terms and obtain
%\begin{align*}
%\frac{\|  e^n \|_{L^2\II} -  \|  e_h^{n-1} \|_{L^2\II} }{\tau} \le  C \tau^{k}  + C  \sum_{j=1}^k \|  e^{n-j}  \|_{L^2\II}.
%\end{align*}
Finally, the desired assertion follows immediately by using discrete Gronwall's inequality
\begin{align*}
\|  e ^n \|_{L^2\II}  \le  C e^{cT} \tau^{k}  + C e^{cT}  \sum_{j=0}^{k-1} \|  e^{j}  \|_{L^2\II}.
\end{align*}
\end{proof}

% {\color{red}
% In fact, in condition that the real solution is sufficiently smooth on space, we do not need high order strict accuracy. Assumption (P2) yields that
% $$ \sum_{i=1}^m c_i^j p_i(\lambda) - \frac{j!}{(-\lambda)^{j+1}}\Big(\sigma(\lambda)
% - \sum_{\ell=0}^j \frac{(-\lambda)^\ell}{\ell!}\Big) =O(\lambda^{k-j}),\quad\text{for all}~ 0\le j\le k-1, \lambda\rightarrow0. $$
% Assume that now we have strictly accurate of orde $k_s<k$. Therefore
% $$- 1  +  j \sum_{i=1}^m c_i^{j-1} p_i(\lambda)
% +\lambda\sum_{i=1}^m c_i^{j} p_i(\lambda)=
% \begin{cases}
%     0, &\mbox{~if~} j\leq k_s-1 \\
%     O(\lambda^{k-j+1}) \mbox{~as~} \lambda\rightarrow0, &\mbox{~if~} k_s\leq j\leq k-1,
% \end{cases}$$
% and
% $$-1 + k \sum_{i=1}^m c_i^{k-1} p_i(\lambda) = O(\lambda)
% \mbox{~as~} \lambda\rightarrow0.
% $$
% Since they tend to $0$ when $\lambda\rightarrow+\infty$, we know that they are bounded. Therefore
% \begin{align*}
%     \|  II_2 \|_{L^2\II} \le C \tau^{k+1} \sum_{j=1}^{k-k_s}\|  \Delta^{j+1} u^{(k)}(t_{n-1}) \| \le  C \tau^{k+1} \sum_{j=1}^{k-k_s}\|  \Delta^{j+1} u \|_{C^{k}([t_{n-1},t_n];L^2\II)},
% \end{align*}
%     and
% \begin{align*}
%         \|  II_3 \|_{L^2\II} \le C \tau^{k+1} \|  \Delta u^{(k)}(t_{n-1}) \| \le  C \tau^{k+1} \|  \Delta u \|_{C^{k}([t_{n-1},t_n];L^2\II)}.
% \end{align*}
% }

\begin{remark}\label{remark:p3}
Theorem \ref{thm:semi} implies that the cut-off operation preserves the maximum bound
without losing global accuracy. However, the Assumption (P3) is restrictive. It is well-known that a single step method
with a given $m \in \mathbb{Z}^+$ could be accurate of order $2m$ (Gauss--Legendre method) \cite[Section 2.2]{Ehle:thesis},
but at most strictly accurate of order $m+1$ \cite[Lemma 5]{BrennerCrouzeixThomee:1982}.
In general, a collocation-type method is only strictly accurate of order $m+1$.

Without the assumption of strict accuracy,  one may still show the error estimate,
provided that $f(u)$ satisfies certain compatibility conditions, e.g.,
$$ f(u) \in C^{\ell}([0,T]; \text{Dom}(\Delta^{k-\ell}))\quad \text{for all} \quad \ell=1,2,\ldots,k,$$
that requires
 $\partial_{\bf n} \Delta^q f(u) = 0$ for $\ell=1,2,\ldots,k-1$.
Unfortunately, those compatibility conditions cannot be fulfilled in general for semilinear parabolic problems.
%The strict accuracy is applied here to avoid those artificial compatibility conditions.
\end{remark}

\begin{remark}\label{remark:p3-2}
The same error estimate could be proved by
assuming that the scheme satisfies the assumption (P3)
with $q=k-1$ and some additional conditions
(see e.g. \cite[Theorem 8.4]{Thomee:2006} and \cite{OstermannRoche:1992}).
However, the proof is not directly applicable when we apply the cut-off operation at each time step.
It warrants further investigation to show the sharp convergence rate $O(\tau^k)$ with weaker assumptions.
\end{remark}

\subsection{Fully discrete scheme}\label{ssec:fully}
In this part, we discuss the fully discrete scheme. To illustrate the
main idea, we consider the one-dimensional case $\Omega=[a,b]$,
and the argument could be straightforwardly extended to multi-dimensional cases, see Remark \ref{rem:highdimension}.
%The extension to multi-dimensional cases is presented in subsection \ref{section:exp-LM-mD}.
We denote by $a=x_0<x_{1}<\dots<x_{Mr}=b$ a partition of the domain with a uniform mesh size $h=x_{ir} - x_{(i-1)r} = (b-a)/M$,
and denote by $S_h^r$ the finite element space of degree $r\ge 1$, i.e.,
$$
S_h^r=\{v\in H^1(\Omega): v|_{I_i}\in P_r,\,\,\, i=1,\dots,M\} ,
$$
where $I_i=[x_{(i-1)r},x_{ir}]$ and $P_r$ denotes the space of polynomials of degree $\le r$.

Let $x_{(i-1)r+j}$ and { $\omega _{j}$}, $j=0,\dots,r$, be the quadrature
points and weights of the $(r+1)$-point Gauss--Lobatto quadrature on the subinterval $I_i$, and denote
\begin{align*}
w_{(i-1)r+j} =
\left\{
\begin{aligned}
&\omega_j &&\mbox{for}\,\,\, 1\le j\le r-1,\\
&2\omega_j &&\mbox{for}\,\,\, j=0, r.
\end{aligned}
\right.
\end{align*}
Then  we consider the piecewise Gauss--Lobatto quadrature approximation of the inner product, i.e.,
$$
(f,g)_h
:=  \sum_{j=0}^{Mr} w_j f(x_j) g(x_j)  .
$$
This discrete inner product induces a norm
$$
\|f_h\|_h=\sqrt{(f_h,f_h)_h} \quad\forall\, f_h\in S_h^r.
$$

%Note that there are $Mr+1$ quadrature points in the closed interval $[a,b]$.
%For $(Mr+1)$-dimensional vectors ${\bf u}$ and ${\bf v}$, we define the discrete $L^2$ inner product and discrete $L^2$ norm by
%$$
%({\bf u},{\bf v})_h=  {\bf M}{\bf u}\cdot {\bf v}
%\quad\mbox{and}\quad
%\|{\bf v}\|_{h}
%= \sqrt{({\bf v},{\bf v})_h} \, ,
%$$
%where ${\bf M}$ is the $(Mr+1)\times (Mr+1)$ mass matrix, a diagonal matrix consisting of the weights of the piecewise Gauss--Lobatto quadrature corresponding to the quadrature points.
Then we have the following lemma for norm equivalence. The proof follows directly from the positivity of
{Gauss--Lobatto} quadrature weights \cite[p. 426]{Quarteroni:2000}.

\begin{lemma}\label{lem:norm-equivalence}
%If $v_h\in S_h^r$ %and ${\bf v}$ is the $(Mr+1)$-dimensional vector consisting of the nodal values of $v_h$,
The discrete norm $\| \cdot \|_h$ is equivalent to usual $L^2$ norm $\| \cdot \|_{L^2\II}$ in sense that
\begin{align*}%\label{norm-equivalence}
 C_1 \|v_h\|_{L^2(\Omega)} \le \|v_h\|_h \le  C_2 \|v_h\|_{L^2(\Omega)} ,\quad \forall v_h\in S_h^r.
\end{align*}
where $C_1$ and $C_2$ are independent of $h$.
\end{lemma}

To develop the fully discrete scheme, we introduce the discrete Laplacian $-\Delta_h: S_h^r\rightarrow S_h^r$ such that
\begin{align}\label{eqn:disc-lap}
(-\Delta_h v_h, w_h)_h = (\nabla v_h, \nabla w_h)\qquad \text{for all} ~~v_h,w_h\in S_h^r.
\end{align}
Then at $n$-th time level, with given $u_h^{n-k}, \ldots, u_h^{n-1} \in S_h^r$, we find an intermediate solution
$\hat u_h^n\in S_h^r$ such that
\begin{equation}\label{eqn:fully}
 \hat u_h^n = \sigma(-\tau \Delta_h) u_h^{n-1} +\tau  \sum_{i=1}^m p_i(-\tau \Delta_h)
 \Big(\sum_{j=1}^{k} L_j (t_{ni}) \Pi_h f(u_h^{n-j})\Big)
\end{equation}
where $t_{ni}=t_{n-1}+c_i\tau$, and $\Pi_h:C(\overline\Omega)\rightarrow S_h^r$ is the Lagrange interpolation operator.
In order to impose the maximum bound, we apply the cut-off postprocessing: find $u_h^n \in S_h^r$ such that
\begin{equation}\label{eqn:fully-cut}
\begin{split}
u_h^n(x_j)  = \min\big(\max \big(\hat u_h^n(x_j), - \alpha \big) , \alpha \big),  \quad j=0,\dots,Mr.
%\Pi_h \min\big(\max \big(\hat u_h^n, - \alpha \big) , \alpha \big) .
\end{split}
\end{equation}
It is equivalent to
\begin{equation*}
u_h^n  = \Pi_h \min\big(\max \big(\hat u_h^n, - \alpha \big) , \alpha \big) .
\end{equation*}
Essentially, the cut-off operation \eqref{eqn:fully-cut} only works on the finite element nodal points.
%, and hence $u_h^n$ still belongs to $S_h^r$.

Next, we shall prove the optimal error estimate of the fully discrete scheme \eqref{eqn:fully}-\eqref{eqn:fully-cut}.
To this end, we need the following stability estimate of operators $\sigma(-\tau \Delta_h)$ and $p_i(-\tau \Delta_h) $.
\begin{lemma}\label{lem:stable-fully}
Let $\Delta_h$ be the discrete Laplacian defined in \eqref{eqn:disc-lap}, and $\sigma(\lambda)$ and $p_i(\lambda)$ are
rational functions satisfying the Assumption (P1). Then there holds that for all $v_h \in S_h^r$
\begin{equation}\label{eqn:stab-est-1}
 \|  \nabla^q \sigma(-\tau \Delta_h) v_h \|_h \le \| \nabla^q v_h \|_h
\quad \text{and}\quad \|  \nabla^q p_i(-\tau \Delta_h ) v_h \|_h \le C\| \nabla^q v_h \|_h
 \end{equation}
 with $i=1,\ldots,m$ and $q=0,1$.
 Meanwhile,
\begin{equation} \label{eqn:stab-est-2}
 \tau \|  \nabla^q \Delta_h p_i (-\tau \Delta_h) v_h \|_h   \le C \|  \nabla^q v_h \|_h
 \quad  i=1,\ldots,m, ~~q=0,1
  \end{equation}
%Here the generic constant $C$ is independent of $\tau$ and $h$.
\end{lemma}

\begin{proof}
Let $\{(\lambda_j, \fy_j^h)\}_{j=1}^{Mr+1}$ be eigenpairs of $-\Delta_h$, where $\{\fy_j^h\}_{j=1}^{Mr+1}$
forms an orthogonal basis of $S_h^r$
in sense that $ (\fy_i^h,\fy_j^h)_h = \delta_{i,j} $.
Then by the Assumption (P1),
we have for any $v_h \in S_h^r$ and $q=0,1$
\begin{align*}
 \| \nabla^q \sigma(-\tau \Delta_h) v_h \|_h^2 &=  \sum_{j=1}^{Mr+1}(\lambda_j^h)^q|\sigma(\tau\lambda_j)|^2 |(v_h , \fy_j^h)_h|^2\\
& \le \sum_{j=1}^{Mr+1} (\lambda_j^h)^q  |(v_h , \fy_j^h)_h|^2 = \| \nabla^q  v_h \|_h^2.
 \end{align*}
This shows the first estimate. The estimate for $p_i$ follows analogously.

Moreover, the numerator of $p_i(\lambda)$
is of lower degree than its denominator (by Assumption (P1)),
and hence there exists  constants $C_1, C_2>0$ such that
$$  |p_i(\lambda)| \le \frac{C_1}{1+C_2\lambda},\quad \text{for all} ~\lambda>0. $$
Then we derive that  for any $v_h \in S_h^r$ and $q=0,1$
\begin{align*}
\tau^2 \|  \nabla^q \Delta_h p_i (-\tau \Delta_h) v_h \|_h^2
&= \tau^2\sum_{j=1}^{Mr+1}(\lambda_j^h)^{q+2} |p_i(\tau\lambda_j)|^2 |(v_h , \fy_j^h)_h|^2\\
& \le C\tau^2\sum_{j=1}^{Mr+1} \frac{(\lambda_j^h)^{q+2}}{(1+C\tau\lambda_j^h)^2}  |(v_h , \fy_j^h)_h|^2 \\
& \le C \sum_{j=1}^{Mr+1} (\lambda_j^h)^q  |(v_h , \fy_j^h)_h|^2 = C\| \nabla^q  v_h \|_h^2,
\end{align*}
where the constant $C$ only depends on $C_1$ and $C_2$.
This proves the assertion \eqref{eqn:stab-est-2}.
\end{proof}

\begin{lemma}\label{lem:truncation}
Let $v\in H^{2r+2}(\Omega)$ with the homogeneous Neumann boundary condition
and $\fy_h\in S_h^r$. Then we have the following estimate
$$ (\Pi_h \Delta v - \Delta_h \Pi_h v, \fy_h )_h \le Ch^{r+1} \| v \|_{H^{2r+2}} \|\fy_h \|_{H^1\II}. $$%,\quad \text{for all} v. $$
\end{lemma}

\begin{proof}
Using the homogeneous Neumann boundary condition
and  \eqref{eqn:disc-lap}, we obtain
\begin{equation}\label{eqn:consist-fully}
\begin{aligned}
  &\quad(\Pi_h \Delta v - \Delta_h \Pi_h v, \fy_h )_h \\
  &= (\Pi_h \Delta v , \fy_h )_h -  (\Delta_h \Pi_h v, \fy_h )_h\\
  &= \Big(( \Delta v , \fy_h )_h - ( \Delta v , \fy_h )\Big) +  \Big(( \Delta v , \fy_h ) -  (\Delta_h \Pi_h v, \fy_h )_h\Big)\\
    &= \Big(( \Delta v , \fy_h )_h - ( \Delta v , \fy_h )\Big) +  \Big(( \partial_x v , \partial_x \fy_h ) -  (\partial_x \Pi_h v, \partial_x \fy_h )\Big)\\
   %  &= \Big(( \Delta v , \fy_h )_h - ( \Delta v , \fy_h )\Big) +  \Big(( \partial_x v , \partial_x \fy_h ) -  (\partial_x \Pi_h v, \partial_x \fy_h )_h\Big).
\end{aligned}
\end{equation}
Since the $(r + 1)$-point Gauss--Lobatto quadrature on each subinterval $I_i$ is exact for polynomials of degree
$2r - 1$ \cite[pp. 425]{Quarteroni:2000}, employing the Bramble--Hilbert lemma
as well as the inverse inequality, we derive that
\begin{align*}
|( \Delta v , \fy_h )_h - ( \Delta v , \fy_h ) | &= \Big|\sum_{i=1}^M \Big(\sum_{j=0}^{r} \omega_j  (\Delta v   \fy_h)(x_{(i-1)r+j}) - \int_{I_i} (\Delta v ) \fy_h \,\d x\Big) \Big|\\
&\le C h^{2r} \sum_{i=1}^M \| \Delta v   \fy_h  \|_{W^{2r,1}(I_i)} \le  C h^{2r} \sum_{i=1}^M \|   v  \|_{H^{2r+2}(I_i)}  \| \fy_h  \|_{H^{r}(I_i)}\\
&\le Ch^{r+1} \sum_{i=1}^M \|   v  \|_{H^{2r+2}(I_i)}  \| \fy_h  \|_{H^{1}(I_i)} \le Ch^{r+1}  \|   v  \|_{H^{2r+2}\II}  \| \fy_h  \|_{H^{1}\II}.
\end{align*}
Similar argument also leads to the estimate for the second term in \eqref{eqn:consist-fully} for $r\ge2$:
\begin{align*}
%( \partial_x v , \partial_x \fy_h ) -  (\partial_x \Pi_h v, \partial_x \fy_h )&=
 |( \partial_x (v- \Pi_h v), \partial_x \fy_h )|
 &= \Big|\sum_{i=1}^M \int_{I_i}  \partial_x (v- \Pi_h v) \partial_x \fy_h \,\d x\Big|
 =\Big|\sum_{i=1}^M \int_{I_i}    (v- \Pi_h v) \partial_x^2 \fy_h \,\d x\Big|\\
& =\Big|\sum_{i=1}^M \int_{I_i}    v  \partial_x^2 \fy_h \,\d x - \sum_{j=0}^{r} \omega_j  (  v  \partial_x^2 \fy_h)(x_{(i-1)r+j}) \Big|\\
&\le C h^{2r} \sum_{i=1}^M \|  v  \partial_x^2 \fy_h  \|_{W^{2r,1}(I_i)} \le  C h^{2r} \sum_{i=1}^M \|   v  \|_{H^{2r+2}(I_i)}  \| \fy_h  \|_{H^{r}(I_i)}\\
&\le Ch^{r+1} \sum_{i=1}^M \|   v  \|_{H^{2r+2}(I_i)}  \| \fy_h  \|_{H^{1}(I_i)} \le Ch^{r+1}  \|   v  \|_{H^{2r+2}\II}  \| \fy_h  \|_{H^{1}\II}.
\end{align*}
Finally, in case that $r=1$, it is easy to observe that
\begin{align*}
%( \partial_x v , \partial_x \fy_h ) -  (\partial_x \Pi_h v, \partial_x \fy_h )&=
 ( \partial_x (v- \Pi_h v), \partial_x \fy_h )
 &= \sum_{i=1}^M \int_{I_i}  \partial_x (v- \Pi_h v) \partial_x \fy_h \,\d x
 =-\sum_{i=1}^M \int_{I_i}    (v- \Pi_h v) \partial_x^2 \fy_h \,\d x=0.
\end{align*}
\end{proof}

To derive an error estimate for the fully discrete scheme \eqref{eqn:fully}-\eqref{eqn:fully-cut}. We need the following extra assumptions on the
rational function $\sigma(\lambda)$.
\begin{itemize}
\item[{\bf (P4)}] The rational function $\sigma(\lambda)$ satisfies $|\sigma(\lambda)| \rightarrow 0$ as $\lambda\rightarrow\infty$.
\end{itemize}
Note that the Assumption (P4) immediately implies \cite[eq. (7.37)]{Thomee:2006}
\begin{align*}
 |\sigma(\lambda)| \le \frac{1}{1+c_0\lambda} \qquad \text{for any}~~\lambda\ge 0,
\end{align*}
with a generic constant $c_0>0$. This further implies
\begin{align*}
 1-|\sigma(\lambda)|^{-2} \le - 2 c_0 \lambda\qquad \text{for any}~~\lambda\ge 0.
\end{align*}
Therefore, we have for any $v_h\in S_h^r$
\begin{align*}
 \|  \sigma(-\tau\Delta_h) v_h \|_h^2 &= \sum_{j=1}^{Mr+1} |\sigma(\tau\lambda_j)|^2 (v_h, \fy_j^h)_h^2
 = \|  v_h \|_h^2 +  \sum_{j=1}^{Mr+1} (|\sigma(\tau\lambda_j)|^2-1) (v_h, \fy_j^h)_h^2 \\
 &=\|  v_h \|_h^2 + \sum_{j=1}^{Mr+1} (1-|\sigma(\tau\lambda_j)|^{-2}) |\sigma(\tau\lambda_j)|^2(v_h, \fy_j^h)_h^2\\
 &\le \|  v_h \|_h^2 - 2c_0\tau \sum_{j=1}^{Mr+1}  \lambda_j |\sigma(\tau\lambda_j)|^2(v_h, \fy_j^h)_h^2
 = \|  v_h \|_h^2 - 2c_0\tau \| \nabla \sigma(-\tau\Delta_h) v_h\|^2.
\end{align*}

Then we are ready to state following main theorem.

\begin{theorem}\label{thm:fully-error}
Suppose that the Assumptions (P1), (P2) and (P4) are fulfilled, and (P3) holds for $q=k$.
Assume that $|u_0|\le \alpha$ and the maximum principle \eqref{eqn:max-AC} holds,
 and assume that the starting values $u_h^l$, $l=0,\dots,k-1$, are given and
$$
|u_h^l(x_j)|\le \alpha , \quad j=0,\dots,Mr,\quad l=0,\dots,k-1.
$$
Then the   {fully discrete} solution given by \eqref{eqn:fully}-\eqref{eqn:fully-cut} satisfies
\begin{align*}%\label{AC-fully-discrete-max-pr}
|u_h^n(x_j)|\le \alpha , \quad j=0,\dots,Mr,\quad n=k,\dots,N,
\end{align*}
and  {for $n=k,\ldots,N$}
\begin{align*}%\label{fully-err-AC}
  \| u(t_n) - u_h^n \|_{L^2(\Omega)} \le  C(\tau^k + h^{r+1}) + C\sum_{l=0}^{k-1}\| u(t_l) - u_h^l \|_{L^2(\Omega)} ,
\end{align*}
provided that {$u\in C^{k+1}([0,T];C(\bar \Omega)) \cap C^{k}([0,T];\text{Dom}(\Delta)) \cap C^1([0,T];H^{2r+2}(\Omega))$, $f$ is locally Lipschitz continuous and $f(u)\in C^k([0,T];L^2\II)\cap C([0,T];H^{2r+2}\II)$.}
\end{theorem}

\begin{proof}
In $[t_{n-1}, t_{n}]$, we note that $\Pi_h u$ satisfies
\begin{align*}
 \partial_t \Pi_h u(t) - \Delta_h \Pi_h u(t) = \Pi_h f(u(t)) + g_h(t),~~t\in(t_{n-1}, t_{n}], \quad\text{with}~ \Pi_h u(t_{n-1}) ~\text{given},
\end{align*}
and $g_h(t) = ( \Pi_h \Delta - \Delta_h \Pi_h) u(t)$.
Then we define its time stepping approximation $w_h^n$ satisfying
\begin{align*}
 w_h^n = \sigma(-\tau\Delta_h) \Pi_h u(t_{n-1}) + \tau\sum_{i=1}^m p_i(-\tau\Delta_h) \Big(\Pi_h f(u) + g_h\Big)(t_n+c_i\tau).
\end{align*}
Then the argument in Theorem \ref{thm:semi} implies that
\begin{equation*}%\label{eqn:err-wh}
\begin{split}
\| \Pi_h u(t_n) - w_h^n  \|_h &\le C\tau^{k+1} \Big(\sup_{t_{n-1}\le t\le t_n}\|  \Pi_h u^{(k+1)}(t)  \|_h + \sup_{t_{n-1}\le t\le t_n} \|  \Delta_h \Pi_h u^{(k)}(t)  \|_h\Big).
%&\le C\tau^{k+1} \Big(\sup_{t_{n-1}\le t\le t_n}\|  w_h^{(k+1)}(t)  \|_h + \sup_{t_{n-1}\le t\le t_n} \|  \Pi_h f^{(k)}(t)  \|_h\Big).
\end{split}
\end{equation*}
The first term of the right hand side is bounded by $\| u \|_{C^{k+1}([0,T]; C(\bar\Omega))}$, while the second one is bounded  as
\begin{align*}
\|  \Delta_h \Pi_h u^{(k)}(t)  \|_h &= \sup_{\fy_h\in S_h^r}\frac{( \Delta_h \Pi_h u^{(k)}(t), \fy_h)_h}{ \|\fy_h\|_h}\\
&=\sup_{\fy_h\in S_h^r}\frac{( \nabla (\Pi_h u^{(k)}(t) - u^{(k)}(t)),\nabla \fy_h) + ( \nabla  u^{(k)}(t),\nabla \fy_h)}{ \|\fy_h\|_h}\\
& \le C h^{-1} \| \nabla (\Pi_h u^{(k)}(t) - u^{(k)}(t))  \|_{L^2\II} + \|  \Delta  u^{(k)}(t) \|_{L^2\II}\le C \|  u^{(k)} \|_{H^2\II}.
\end{align*}
Therefore, we conclude that
\begin{equation*}
\begin{split}
\| \Pi_h u(t_n) - w_h^n  \|_h &\le C\tau^{k+1} \Big( \| u \|_{C^{k+1}([t_{n-1},t_{n}]; C(\bar\Omega))} +   \|  u \|_{C^{k}([t_{n-1},t_{n}]; H^2(\Omega))}\Big).
%&\le C\tau^{k+1} \Big(\sup_{t_{n-1}\le t\le t_n}\|  w_h^{(k+1)}(t)  \|_h + \sup_{t_{n-1}\le t\le t_n} \|  \Pi_h f^{(k)}(t)  \|_h\Big).
\end{split}
\end{equation*}
Then the simple triangle inequality leads to
\begin{equation}\label{eqn:split}
\begin{split}
\| \hat u_h^n - \Pi_h u(t_n)  \|_h^2 &\le \Big(\| \hat u_h^n - w_h^n \|_h + \|w_h^n - \Pi_h u(t_n)   \|_h \Big)^2 \\
&\le (1+C\tau) \| \hat u_h^n - w_h^n \|_h^2 + C\tau^{2k+1}.
\end{split}
\end{equation}
Let $\rho_h^n = \hat u_h^n  - w_h^n $ and $e_h^n = u_h^n - \Pi_h u(t_n)$, then $\rho_h^n$ satisfies
\begin{align}\label{eqn:rhoh}
\rho_h^n = \sigma(-\tau \Delta_h) e_h^{n-1} + I_1^n + I_2^n
%, \tau \sum_{i=1}^m p_i(-\tau \Delta_h) \Big((I_\tau^k \Pi_h f) (t_{n-1}+c_i\tau) - \Pi_h f(u(t_{n-1}+c_i\tau)) \Big)
%- \tau \sum_{i=1}^m p_i(-\tau \Delta_h) g_h(t_{n-1}+c_i\tau).
\end{align}
where
\begin{align*}
&  I_1^n = \tau \sum_{i=1}^m p_i(-\tau \Delta_h) \Big(
\sum_{j=1}^{k} L_j (t_{n-1} + c_i\tau) \Pi_h f(u_h^{n-j}) - \Pi_h f(u(t_{n-1}+c_i\tau)) \Big) ,\\
& \text{and}~~
I_2^n = - \tau \sum_{i=1}^m p_i(-\tau \Delta_h) g_h(t_{n-1}+c_i\tau).
\end{align*}
Now take the discrete inner product between \eqref{eqn:rhoh} and $\rho_h^n$
\begin{align*}
 \| \rho_h^n \|_h^2 = (\sigma(-\tau \Delta_h) e_h^{n-1}, \rho_h^n)_h + (I_1^n,\rho_h^n)_h + (I_2^n,\rho_h^n)_h.
\end{align*}
Then first term, we apply the Assumption (P4) to obtain that
\begin{align*}
(\sigma(-\tau \Delta_h) e_h^{n-1}, \rho_h^n) &\le \frac12\|  \sigma(-\tau \Delta_h) e_h^{n-1} \|_h^2 + \frac12\| \rho_h^n\|_h^2 \\
&\le \frac12\|  e_h^{n-1} \|_h^2 - c_0\tau \| \nabla \sigma(-\tau \Delta_h) e_h^{n-1} \|^2 + \frac12\| \rho_h^n\|_h^2\\
&\le \frac12\|  e_h^{n-1} \|_h^2 - c_0\tau \| \nabla (\rho_h^n - I_1^n - I_2^n) \|^2 + |  |+ \frac12\| \rho_h^n\|_h^2\\
&\le \frac12\|  e_h^{n-1} \|_h^2 - c_0\tau \| \nabla \rho_h^n\|^2 - c_0\tau \| \nabla(I_1^n +I_2^n) \|^2\\
&\qquad+2c_0 \tau (\nabla \rho_h^n,\nabla (I_1^n +I_2^n))+ \frac12\| \rho_h^n\|^2
\end{align*}
Then applying the definition of $\Delta_h$, we arrive at
\begin{equation}\label{eqn:rhoh-3}
\begin{aligned}
 \frac12 \| \rho_h^n \|_h^2 %&\le \frac12\|  e_h^{n-1} \|_h^2 - c_0\tau \| \nabla \rho_h^n\|^2
%+ 2c_0\tau (\nabla \rho_h^n,\nabla (I_1^n+I_2^n)) + (I_1^n,\rho_h^n)_h + (I_2^n,\rho_h^n)_h\\
&\le \frac12\|  e_h^{n-1} \|_h^2 - c_0\tau \| \nabla \rho_h^n\|^2\\
&\quad- 2c_0\tau ( \rho_h^n,\Delta_h (I_1^n+I_2^n))_h + (I_1^n,\rho_h^n)_h + (I_2^n,\rho_h^n)_h.
\end{aligned}
\end{equation}
By using the approximation property of interpolation $I_\tau^k$, Lemma \ref{lem:stable-fully},
 and the fact that $u_h^{n-k},\dots,u_h^{n-1}$ satisfies the maximum bound,
we bound the fourth term in \eqref{eqn:rhoh-3}  as
\begin{align*}
 | (I_1^n,\rho_h^n)_h |  &\le   \tau \sum_{i=1}^m \Big|\Big( \sum_{j=1}^k L_j(t_{n-1} + c_i\tau) \Pi_hf(u(t_{n-j}))- \Pi_h f(u(t_{n-1} + c_i\tau), p_i(-\tau \Delta_h)\rho_h^n\Big)_h\Big|   \\
 +&  \tau \sum_{i=1}^m \Big| \Big( \sum_{j=1}^k L_j(t_{n-1} + c_i\tau) (\Pi_h f(u(t_{n-j}))-\Pi_hf(u_h^{n-j})), p_i(-\tau \Delta)\rho_h^n\Big)_h \Big| \\
   &\le  C \tau \sum_{i=1}^m \|p_i(-\tau \Delta_h)\rho_h^n \|_h  \sum_{j=1}^k  \|\Pi_hf(u(t_{n-j}))- \Pi_h f(u^{n-j})\|_h   \\
     &\quad+  C \tau^{k+1} \sum_{i=1}^m  \|p_i(-\tau \Delta_h)\rho_h^n \|_h  \|  \Pi_h f(u)  \|_{C^{k}([t_{n-k},t_n]; L^2\II)} \\
 &\le C\tau^{2k+1} \|  \Pi_h f(u)  \|_{C^{k}([t_{n-k},t_n]; L^2\II)}^2 + C \tau \sum_{j=1}^k \|  e_h^{n-j}  \|_h^2 +  C\tau \|  \rho_h^n \|_h^2 \\
  &\le C\tau^{2k+1} \|  f(u)  \|_{C^{k}([t_{n-k},t_n]; C(\bar\Omega))}^2 + C \tau \sum_{j=1}^k \|  e_h^{n-j}  \|_h^2 +  C\tau \|  \rho_h^n \|_h^2.
\end{align*}
The fifth term in \eqref{eqn:rhoh-3} can be bounded by using lemmas \ref{lem:stable-fully} and \ref{lem:truncation}, i.e.,
\begin{equation}\label{eqn:fully-I2}
\begin{split}
| (I_2^n,\rho_h^n)_h | &\le C \tau \sum_{i=1}^m |(g_h(t_{n-1}+c_i\tau), p_i(-\tau\Delta_h)\rho_h^n)_h| \\
&\le C\tau \sum_{i=1}^m h^{r+1} \| u(t_{n-1}+c_i\tau)  \|_{H^{2r+2}(\Omega)} \| p_i(-\tau\Delta_h) \rho_h^n \|_{H^1\II} \\
&\le \frac{C \tau h^{2r+2}}{\eta}  \| u  \|_{C([t_{n-1},t_n];H^{2r+2}(\Omega))}^2 + C\tau \eta  \| \rho_h^n \|_{H^1\II}^2.
\end{split}
 \end{equation}
For the third term in the right hand side of \eqref{eqn:rhoh-3}, we shall apply the preceding argument again, together with the stability estimate
\eqref{eqn:stab-est-2}, and obtain that
\begin{equation}\label{eqn:fully-I3}
\begin{split}
\tau |( \rho_h^n,\Delta_h (I_1^n+I_2^n))_h| &\le  C \tau^2 \sum_{i=1}^m \|\Delta_h p_i(-\tau \Delta_h)\rho_h^n \|_h  \sum_{j=1}^k  \|\Pi_hf(u(t_{n-j}))- \Pi_h f(u^{n-j})\|_h   \\
     &\quad+  C \tau^{k+2} \sum_{i=1}^m  \| \Delta_hp_i(-\tau \Delta_h)\rho_h^n \|_h  \|  \Pi_h f(u)  \|_{C^{k}([t_{n-k},t_n]; L^2\II)} \\
&\quad+ C\tau^2 \sum_{i=1}^m h^{r+1} \| u(t_{n-1}+c_i\tau)  \|_{H^{2r+2}(\Omega)} \| \Delta_h p_i(-\tau\Delta_h) \rho_h^n \|_{H^1\II} \\
%&\le  C \tau \sum_{i=1}^m \|   \rho_h^n \|_h  \sum_{j=1}^k  \|\Pi_hf(u(t_{n-j}))- \Pi_h f(u^{n-j})\|_h   \\
%     &\quad+  C \tau^{k+1} \sum_{i=1}^m  \|   \rho_h^n \|_h  \|  \Pi_h f(u)  \|_{C^{k}([t_{n-k},t_n]; L^2\II)} \\
%&\quad+ C\tau \sum_{i=1}^m h^{r+1} \| u(t_{n-1}+c_i\tau)  \|_{H^{2r+2}(\Omega)} \|   \rho_h^n \|_{H^1\II} \\
 &\le C\tau^{2k+1} \|  f(u)  \|_{C^{k}([t_{n-k},t_n]; C(\bar\Omega))}^2 + C \tau \sum_{j=1}^k \|  e_h^{n-j}  \|_h^2 +  C\tau \|  \rho_h^n \|_h^2\\
&\quad+\frac{C \tau h^{2r+2}}{\eta}  \| u  \|_{C([t_{n-1},t_n];H^{2r+2}(\Omega))}^2 + C\tau \eta  \| \rho_h^n \|_{H^1\II}^2.
\end{split}
 \end{equation}
Then by choosing $\eta$ small, we arrive at
\begin{align*}
(1-C\tau) \| \rho_h^n \|_h^2 &\le  \|  e_h^{n-1} \|_h^2   + C \tau \sum_{j=1}^k  \|  e_h^{n-j} \|_h^2
+ C\tau (\tau^{2k} + h^{2r+2}).
\end{align*}
This together with \eqref{eqn:split} and the property of the cut-off operation lead to
\begin{align*}
\| e_h^n  \|_h^2 &\le \| \hat u_h^n - \Pi_h u(t_n) \|_h^2   \le (1+C\tau) \| \rho_h^n \|_h^2 + c\tau^{2k+1}\\
&\le\| e_h^{n-1}  \|_h^2 + C \tau \sum_{j=1}^k  \|  e_h^{n-j} \|_h^2  + C\tau (\tau^{2k} + h^{2r+2}),
\end{align*}
and hence we rearrange terms and obtain
\begin{align*}
\frac{\|  e_h^n \|_h^2 -  \|  e_h^{n-1} \|_h^2 }{\tau} \le  C (\tau^{2k} + h^{2r+2})  + C  \sum_{j=1}^k \|  e_h^{n-j}  \|_h^2.
\end{align*}
Then the discrete Gronwall's inequality implies
\begin{align*}
\|  e_h^n \|_h^2  \le  C e^{cT} (\tau^{2k}+h^{2r+2})  + C e^{cT}  \sum_{j=0}^{k-1} \|  e_h^{j}  \|_h^2,
\end{align*}
and the desired error estimate follows from the equivalence of different norms by Lemma \ref{lem:norm-equivalence}.
\end{proof}

\begin{remark}\label{rem:optimal-LiYangZhou}
In \cite{LiYangZhou:sisc}, an error estimate $O(\tau^k + h^r)$,
which is  suboptimal in space, was derived
for the multistep exponential integrator method by using energy argument.
The loss of the optimal convergence rate is due to the
suboptimal estimate of the term $(\partial_x(\Pi_h u  -  u), \partial_x v_h)$
in \cite[eq. (2.6) and (3.22)]{LiYangZhou:sisc}. The optimal rate could be also
proved by using Lemma \ref{lem:truncation}.

The Assumption (P4) , called
\textsl{L-stability}, is useful when solving stiff problems.
It is also essential in the proof of Theorem \ref{thm:fully-error} to derive the optimal error estimate
of the extrapolated cut-off single step scheme. In particular,
Assumption (P4) immediately leads to the estimate
\begin{align*}
 \|  \sigma(-\tau\Delta_h) v_h \|_h^2 \le \|  v_h \|_h^2 - 2c_0\tau \| \nabla \sigma(-\tau\Delta_h) v_h\|^2,
\end{align*}
where the second term in the right side is used to handle the term involving $\|  \rho_h^n \|_{H^1\II}$
in \eqref{eqn:fully-I2} and \eqref{eqn:fully-I3}.
Many single step methods, e.g., Lobatto IIIC and Radau IIA methods are L-stable \cite{Ehle:thesis,HairerWanner:2010}.
For both classes, arbitrarily high-order methods can be constructed.
Nevertheless, it is not clear how to remove the restriction (P4) in general.
\end{remark}

\begin{remark}\label{rem:highdimension}
It is straightforward to extend the argument to higher dimensional problems, e.g., $\Omega$ is
a multi-dimensional rectangular domain $(a, b)^d \subset\R^d$, with $d\ge 2$.
Then we can divide $\Omega$ in to some small sub-rectangles, called partition $\mathcal{K}$,
and  apply the tensor-product Lagrange finite elements on the partition $\mathcal{K}$.
As a result, Lemma \ref{lem:truncation} is still valid, which implies the desired error estimate.
See more details about the setting for multi-dimensional problems in \cite[Section 2.2]{LiYangZhou:sisc}.
\end{remark}

\section{Collocation-type methods with the cut-off postprocessing}\label{sec:collocation}
Note that the Assumption (P4) excludes some
popular methods, e.g., Gauss--Legendre methods.
This motivates us to discuss the collocation-type schemes,
which belong to implicit Runge--Kutta methods, and derive error estimate without Assumption (P4).
This class of time stepping methods is easy to implement,
and plays an essential role in the next section to develop an energy-stable scheme.
For simplicity, we only present the argument for one-dimensional case,
and it can be extended to multi-dimensional cases straightforwardly as mentioned in Remark \ref{rem:highdimension}.

\begin{table}[h]
\begin{center}
\begin{tabular}{lll|ll}
$a_{11}$  &\ldots &$a_{1m}$ &$c_1$  &  \\
\vdots    &       &\vdots   &\vdots &  \\
$a_{m1}$  &\ldots &$a_{mm}$ &$c_m$  &  \\ \cline{1-4}
$b_{1}$   &\ldots & $b_{m}$ & &
\end{tabular}
\end{center}
\caption{Butcher tableau for Runge--Kutta scheme.}\label{tab:RK}
\end{table}

Now we consider an $m$-stage Runge--Kutta method, described by the Butcher tableau \ref{tab:RK}. %\red{more discription...}
Here $\{c_i\}_{i=1}^m$ denotes $m$ distinct quadrature points.
\begin{definition}
We call a Runge--Kutta method is algebraically stable if the method satisfies
\begin{itemize}
\item[{\bf (P5)(a)}] The matrix $A = (a_{ij})$, with $i,j=1,\ldots,m$ is invertible;
\item[{\bf (P5)(b)}] The coefficients $b_i$ satisfy $b_i>0$ for $i=1,2,\dots,m$;
\item[{\bf (P5)(c)}] The  symmetric matrix $\mathcal{M}\in  \mathbb{R}^{m\times m }$
with entries $m_{ij} := b_ia_{ij} + b_ja_{ji} - b_ib_j$, $i, j = 1,\ldots,m$
is positive semidefinite. %This is our essential condition for the energy decay property.
\end{itemize}
\end{definition}

Here we assume that the Runge--Kutta scheme described  by Table \ref{tab:RK}
associates with a collocation method, i.e., coefficients $a_{ij}, b_i, c_i$ satisfy
\begin{align}
        \sum_{i=1}^m b_i c_i^{l-1} &= \frac1l, \quad l = 1, \cdots, p,\label{eqn:accuracy} \\
        \sum_{j=1}^m a_{ij} c_j^{l-1} &= \frac{c_i^l}l, \quad l = 1, \cdots, m,\label{eqn:stage-order}
\end{align}
with some integers $p\ge m$. Two popular families of algebraically stable Runge--Kutta
methods of collocation type satisfying (2.6) of orders $p = 2m$ and $p = 2m - 1$
 are the Gauss--Legendre  methods and the Radau IIA methods respectively. For both classes,
arbitrarily high order methods can be constructed. Note that the Gauss--Legendre methods
are not L-stable \cite{HairerWanner:2010}.

%The one-stage members of these families are the midpoint (or Crank?Nicolson) and implicit Euler methods, respectively.

In particular, at level $n$, with given $u_h^{n-k},\ldots, u_h^{n-1}\in S_h^r$, we find an intermediate solution $\hat u_h^{n}\in S_h^r$ such that
\begin{equation}\label{eqn:RK-fully}
    \begin{cases}
    \dot u_h^{ni} = \Delta_h u_h^{ni} +  \sum_{\ell=1}^k L_{\ell}(t_{n-1}+c_i\tau)\Pi_h f(u_h^{n-\ell}) &\quad \text{for} ~~i=1,2,\dots,m, \\ %\sum_{\ell=1}^m L_{n-\ell}(t_{n-1}+c_i\tau)\Pi_h W(u_h^{n-j}),\\
    u_h^{ni} = u_h^{n-1} + \tau \sum_{j=1}^m a_{ij} \dot u_h^{nj} &\quad \text{for} ~~i=1,2,\dots,m, \\
    \hat u_h^{n} = u_h^{n-1} + \tau \sum_{i=1}^m b_i \dot u_h^{ni},
  \end{cases}
\end{equation}
where $k=\min(p,m+1)$, and $\Pi_h:C(\overline\Omega)\rightarrow S_h^r$ is the Lagrange interpolation operator.
Then we apply the cut-off operation: find $u_h^n\in S_h^r$ such that
\begin{equation}\label{eqn:RK-cut}
u_h^n(x_j)  = \min\big(\max \big(\hat u_h^n(x_j), - \alpha \big) , \alpha \big),  \quad j=0,\dots,Mr.
\end{equation}

\begin{remark}\label{rem:equiv}
Note
that the scheme \eqref{eqn:RK-fully} is equivalent to \eqref{eqn:fully} with
$$  (p_1(\lambda),\ldots,p_m(\lambda))= (b_1, \ldots, b_m) (I+\lambda A)^{-1},\qquad
 \sigma(\lambda) = 1-\lambda \sum_{j=1}^m b_j p_j(\lambda). $$
Then the Assumption (P5), and \eqref{eqn:accuracy}-\eqref{eqn:stage-order}
imply Assumptions (P1), (P2) with order $k=\min(p,m+1)$ and (P3) with order $q=\min(p,m+1)$.
Hence Theorem \ref{thm:fully-error} indicates the temporal error $O(\tau^{\min(p,m+1)})$.
This is the reason why we choose $k$-step extrapolation, where $k=\min(p,m+1)$, in the time stepping scheme \eqref{eqn:RK-fully}.
\end{remark}

Next, we shall derive an error estimate for the fully discrete scheme \eqref{eqn:RK-fully}-\eqref{eqn:RK-cut}.
To begin with, we shall examine the local truncation error. We define the local truncation error $\eta_{ni}$ and $\eta_{n+1}$  as
\begin{equation}\label{eqn:RK-fully-trunc}
    \begin{cases}
    \dot u_*^{ni} = \Delta  u(t_{ni}) +
    \sum_{\ell=1}^k L_{\ell}(t_{ni}) f(u(t_{n-\ell}))&\quad \text{for} ~~i=1,2,\dots,m, \\ %\sum_{\ell=1}^m L_{n-\ell}(t_{n-1}+c_i\tau)\Pi_h W(u_h^{n-j}),\\
    u(t_{ni})=  u(t_{n-1}) + \tau \sum_{j=1}^m a_{ij} \dot u_*^{nj} + \eta_{ni}&\quad \text{for} ~~i=1,2,\dots,m, \\
    u(t_n) = u(t_{n-1}) + \tau \sum_{i=1}^m b_i \dot u_*^{ni} + \eta_{n}
  \end{cases}
\end{equation}
where $t_{ni}=t_{n-1}+c_i\tau$ and $k=\min(p,q+1)$.
Then the next lemma give an estimate for the local truncation error $\eta_{ni}$ and $\eta_{n}$.
We sketch the proof in Appendix for completeness.

\begin{lemma}\label{lem:RK-local-truncation}
Suppose that the Assumption (P5), and relations \eqref{eqn:accuracy} and \eqref{eqn:stage-order}
are valid. Then the local truncation error $\eta_{ni}$ and $\eta_{n}$, given by \eqref{eqn:RK-fully-trunc},
satisfy the estimate
\begin{equation*}
\| \eta_n \|_{H^1\II}  + \tau \sum_{i=1}^m  \| \eta_{ni} \|_{H^1\II}   \le C \tau^{k+1}.
\end{equation*}
with $k=\min(p,q+1)$, provided that $u\in C^{k+1}([0,T]; H^1\II)$ and $f(u) \in C^k([0,T]; H^1\II)$.
\end{lemma}

Then we are ready to present the following theorem, which gives the error estimate for the
cut-off Runge--Kutta scheme \eqref{eqn:RK-fully}-\eqref{eqn:RK-cut}.

\begin{theorem}\label{thm:RK-error}
Suppose that the Runge--Kutta method given by Table \ref{tab:RK} satisfies Assumption (P5),
and relations \eqref{eqn:accuracy} and \eqref{eqn:stage-order}  are valid.
Assume that $|u_0| \le \alpha$ and the maximum principle \eqref{eqn:max-AC} holds,
 and assume that the starting values $u_h^n$, $l=0,\dots,k-1$, are given and
$$
|u_h^l(x_j)|\le \alpha , \quad j=0,\dots,Mr,\quad l=0,\dots,k-1.
$$
Then the   {fully discrete} solution given by \eqref{eqn:RK-fully}-\eqref{eqn:RK-cut} satisfies
\begin{align*}
|u_h^n(x_j)|\le \alpha , \quad j=0,\dots,Mr,\quad n=k,\dots,N,
\end{align*}
and  {for $n=k,\ldots,N$}
\begin{align*}
  \| u(t_n) - u_h^n \|_{L^2(\Omega)} \le  C(\tau^k + h^{r+1}) + C\sum_{l=0}^{k-1}\| u(t_l) - u_h^l \|_{L^2(\Omega)} ,
\end{align*}
provided that {$u\in C^{k+1}([0,T]; H^1\II) \cap C^1([0,T];H^{2r+2}(\Omega))$,
$f$ is locally Lipschitz continuous and $f(u)\in C^k([0,T];H^1\II)\cap C([0,T];H^{2r+2}\II)$.}
%provided that $u, f$ and $f(u)$ are sufficiently smooth in both time and space variables.
%\in C^{k+1}([0,T];C(\bar \Omega)) \cap C^{k}([0,T];\text{Dom}(\Delta)) \cap C^1([0,T];H^{4}(\Omega))$,
%$f$ is locally Lipschitz continuous and $f(u)\in C^k([0,T];L^2\II)\cap C([0,T];H^{4}\II)$.
\end{theorem}

\begin{proof}
Due to the cut-off operation \eqref{eqn:semi-cut}, the discrete
maximum bound principle follows immediately.
With the notation
\begin{align*}
  e_h^{ni} = \Pi_h u(t_{ni}) - u_h^{ni},\quad  \dot e_h^{ni} = \Pi_h \dot u_*^{ni} - \dot u_h^{ni}, \quad
  e_h^{n} = \Pi_h u(t_n) - u_h^{n},  \quad  \hat e_h^{n} = \Pi_h u(t_n) - \hat u_h^{n},
\end{align*}
we derive the error equations
\begin{equation}\label{eqn:RK-error-eq}
    \begin{cases}
   \dot e_h^{ni} =  \Delta_h e_h^{ni}  +  (\Pi_h \Delta- \Delta_h \Pi_h) u(t_{ni}) +
    \sum_{\ell=1}^k L_{\ell}(t_{ni}) \Pi_h (f(u(t_{n-\ell})) - f(u_h^{n-\ell}))&\quad \text{for} ~~i=1,2,\dots,m, \\ %\sum_{\ell=1}^m L_{n-\ell}(t_{n-1}+c_i\tau)\Pi_h W(u_h^{n-j}),\\
    e_h^{ni}=  e_h^{n-1} + \tau \sum_{j=1}^m a_{ij} \dot e_h^{nj} + \Pi_h \eta_{ni}&\quad \text{for} ~~i=1,2,\dots,m, \\
    \hat e_h^n = e_h^{n-1} + \tau \sum_{i=1}^m b_i \dot e_h^{ni} + \Pi_h\eta_{n} .
  \end{cases}
\end{equation}
Take the square of discrete $L^2$ norm of both sides of the last relation of  \eqref{eqn:RK-error-eq}, we obtain
\begin{equation}\label{eqn:rk-split-01}
\begin{split}
\|\hat e_h^{n}\|_h^2 & = \|e_h^{n-1}+\tau\sum_{i=1}^{m} b_i\dot{e}_h^{ni}\|_h^2
+ 2(\eta_{n}, e_h^{n-1}+ \tau\sum_{i=1}^{m} b_i\dot{e}_h^{ni})_h + \|\Pi_h \eta_{n}\|_h^2.
\end{split}
\end{equation}
For the first term on the right hand side,
we expand it and apply the second equation of \eqref{eqn:RK-error-eq}  to obtain
\begin{align*}
\|e_h^{n-1}+\tau\sum_{i=1}^{m} b_i\dot{e}_h^{ni}\|_h^2
&= \|e_h^{n-1}\|_h^2 + 2\tau\sum_{i=1}^{m}b_i(\dot{e}_h^{ni},e_h^{ni}-\eta_{ni})_h - \tau^2\sum_{i,j=1}^{m} m_{ij}(\dot{e}_h^{ni}, \dot{e}_h^{nj})_h\\
&\leq \|e_h^{n-1}\|_h^2 + 2\tau\sum_{i=1}^{m}b_i(\dot{e}_h^{ni},e_h^{ni}-\eta_{ni})_h,
\end{align*}
where in the last inequality we use the positive semi-definiteness of the matrix $\mathcal{M}$ in the Assumption (P5).
Next, we note that the first relation of \eqref{eqn:RK-error-eq} implies
\begin{equation*}
\begin{split}
(\dot{e}_h^{ni},e_h^{ni}-\eta_{ni})_h
&= \Big(   \Delta_h e_h^{ni} +   \sum_{\ell=1}^k L_{\ell}(t_{ni})   (f(u(t_{n-\ell})) - f(u_h^{n-\ell}))
+ (\Pi_h \Delta - \Delta_h \Pi_h) u(t_{n-1}), e_h^{ni}-\eta_{ni} \Big)_h\\
&= -\|\nabla e_h^{ni}\|_{L^2\II}^2 + (\nabla e_h^{ni} , \nabla \Pi_h \eta_{ni}) +  \Big(\sum_{\ell=1}^k L_{\ell}(t_{ni})   (f(u(t_{n-\ell})) - f(u_h^{n-\ell})), e_h^{ni}-\eta_{ni} \Big)_h  \\
& \quad \\
& \quad +  \Big(  (\Pi_h \Delta - \Delta_h \Pi_h) u(t_{n-1}), e_h^{ni}-\eta_{ni}\Big)_h
\end{split}
\end{equation*}
The bound of second term of the right hand side can be derived via Cauchy-Schwarz inequality
\begin{equation*}%\label{eqn:est-rk-01}
|(\nabla e_h^{ni} , \nabla \Pi_h \eta_{ni})| \le \frac14 \|  \nabla e_h^{ni} \|_{L^2\II}^2 + C \| \eta_{ni}\|_{H^1\II}^2.
\end{equation*}
Meanwhile, using the fact that $f$ is locally Lipschitz and the fully disctete solutions
satisfy maximum bound principle at the Gauss--Lobatto points, the third term can be bounded as
\begin{equation*}
\begin{split}
\Big(\sum_{\ell=1}^k L_{\ell}(t_{ni})   (f(u(t_{n-\ell})) - f(u_h^{n-\ell})), e_h^{ni}-\eta_{ni}\Big)_h
 \le C \Big( \|  e^{ni}_h \|_h^2 + \|\eta_{ni} \|_{H^1\II}^2 + \sum_{\ell=1}^k \| e_h^{n-\ell} \|_h^2 \Big)
\end{split}
\end{equation*}
The bound of the last term follows from Lemma \ref{lem:truncation}
\begin{equation*}%\label{eqn:est-rk-01}
\begin{split}
 \Big(  (\Pi_h \Delta - \Delta_h \Pi_h) u(t_{n-1}), e_h^{ni}- \eta_{ni} \Big)_h &\le C h^{r+1}  \| e_h^{ni}- \Pi_h  \eta_{ni}   \|_{H^1\II} \\
 &\le  \frac14\| \nabla e_h^{ni}  \|_{L^2\II}^2 + C(\|  e_h^{ni} \|_h^2 + \|  \eta_{ni} \|_{H^1\II}^2+ h^{2r+2} ).
\end{split}
\end{equation*}
Therefore, we arrive at
\begin{equation*}
2(\dot{e}_h^{ni},e_h^{ni}-\eta_{ni})_h \le  -\|\nabla e_h^{ni}\|_{L^2\II}^2 +
 C \Big(  \sum_{j=1}^k \| e_h^{n-j} \|_h^2 + \|  e_h^{ni} \|_h^2  + \| \eta_{ni}\|_{H^1\II}^2 + h^{2r+2} \Big),
\end{equation*}
and hence by Lemma \ref{lem:RK-local-truncation}, we derive
\begin{align*}
\|e_h^{n-1}+\tau\sum_{i=1}^{m} b_i\dot{e}_h^{ni}\|_h^2
&\leq \|e_h^{n-1}\|_h^2 - \tau\sum_{i=1}^{m}b_i \| \nabla e_h^{ni}  \|_{L^2\II}^2 + C\tau \sum_{i=1}^m  \|  e_h^{ni} \|_h^2 \\
&\qquad + C\tau \sum_{j=1}^k \| e_h^{n-j} \|_h^2 + C \tau(h^{2r+2}+\tau^{2k}).
\end{align*}
In view of the first relation of the error equation \eqref{eqn:RK-error-eq}, we have the estimate
\begin{align*}
(\eta_{n}, e_h^{n-1}+ \tau\sum_{i=1}^{m} b_i\dot{e}_h^{ni})_h
&\le \| \eta_n  \|_{H^1\II} \Big( \|e_h^{n-1}\|_h + C \tau  \sum_{i=1}^m b_i
\Big(\| \nabla e_h^{ni} \|_h + \sum_{j=1}^{k} \|  e_h^{n-j} \|_h  + h^{2r+2}\Big)\Big)\\
&\le C\tau(h^{2r+2}+ \tau^{2k})   + \frac{\tau}{4}   \sum_{i=1}^{m} b_i  \| \nabla e_h^{ni} \|_h^2
+ C\tau  \sum_{j=1}^{k} \|  e_h^{n-j} \|_h^2
\end{align*}
which gives a bound of the second term in \eqref{eqn:rk-split-01}. In conclusion, we obtain that
\begin{equation}\label{eqn:rk-est-02}
\begin{split}
 \|\hat e_h^{n}\|_h^2 + \frac\tau2 \sum_{i=1}^m \| \nabla e_h^{ni} \|_{L^2\II}^2& \le  C\tau(h^4+ \tau^{2k}) +  \|e_h^{n-1}\|_h^2
 + C \tau   \sum_{i=1}^m \|  e_h^{ni} \|_h^2 + C\tau  \sum_{j=1}^{k} \|  e_h^{n-j} \|_h^2.
\end{split}
\end{equation}
Next, we shall derive a bound for  $ \sum_{i=1}^m \|  e_h^{ni} \|_h^2 $
 on the right-hand side. To this
end, we test the second relation of  \eqref{eqn:RK-error-eq} by $e_h^{ni}$. This yields
\begin{equation*}
\begin{split}
 \sum_{i=1}^m \|e_h^{ni}\|_h^2 & \le C \|  e_h^{n-1} \|_h^2 + C\tau\sum_{i,j=1}^m a_{ij}(\dot e_h^{nj},e_h^{ni})_h + C\sum_{i=1}^m \| \Pi_h \eta_{ni} \|_h^2\\
& \le  C \|  e_h^{n-1} \|_h^2 + C\tau\sum_{i,j=1}^m a_{ij}(\dot e_h^{nj},e_h^{ni}) _h + C\tau^{2k}.
\end{split}
\end{equation*}
Then, we apply the first relation of  \eqref{eqn:RK-error-eq} and Lemma \ref{lem:truncation} to derive
\begin{equation*}
\begin{split}
 \sum_{i,j=1}^m a_{ij}(\dot e_h^{nj},e_h^{ni})_h &= - \sum_{i,j=1}^m a_{ij}(\nabla e_h^{nj}, \nabla e_h^{ni})
 + \sum_{i,j=1}^m a_{ij}   \Big( \sum_{\ell=1}^k L_{\ell}(t_{ni})   (f(u(t_{n-\ell})) - f(u_h^{n-\ell})) , e_h^{ni}\Big)_h\\
&\quad + \sum_{i,j=1}^m a_{ij}  ((\Pi_h \Delta - \Delta_h \Pi_h) u(t_{n-1}),  e_h^{ni} )_h \\
&\le C\sum_{i=1}^m \Big(\| \nabla e_h^{ni}  \|_{L^2\II}^2 + \|  e_h^{ni}  \|_{h}^2  \Big)+ Ch^{2r+2} + C\sum_{\ell=1}^k \| e_h^{n-\ell} \|_h^2.
\end{split}
\end{equation*}
Therefore, we obtain
\begin{equation*}
\begin{split}
 \sum_{i=1}^m \|e_h^{ni}\|_h^2   \le  C( \tau h^{2r+2}+\tau^{2k}) + C\|  e_h^{n-1} \|_h^2 + C\tau \sum_{\ell=1}^k \| e_h^{n-\ell} \|_h^2 +
 C\tau \sum_{i=1}^m \Big(\| \nabla e_h^{ni}  \|_{L^2\II}^2 + \|  e_h^{ni}  \|_{h}^2   \Big).
\end{split}
\end{equation*}
Then for sufficiently small $\tau$,
$ C\tau \sum_{i=1}^m \|e_h^{ni}\|_h^2   $
on the right-hand side can be absorbed by the left-hand side. Then, we obtain
\begin{equation*}
\begin{split}
 \sum_{i=1}^m \|e_h^{ni}\|_h^2   \le  C( \tau h^{2r+2}+\tau^{2k}) + C\|  e_h^{n-1} \|_h^2 + C\tau \sum_{\ell=1}^k \| e_h^{n-\ell} \|_h^2 +
 C\tau \sum_{i=1}^m \| \nabla e_h^{ni}  \|_{L^2\II}^2 .
\end{split}
\end{equation*}
Now substituting the above estimate into \eqref{eqn:rk-est-02}, there holds for sufficiently small $\tau$
\begin{equation*}
\begin{split}
 \|\hat e_h^{n}\|_h^2
 \le&  C\tau(h^{2r+2}+ \tau^{2k}) +  \|e_h^{n-1}\|_h^2   + C\tau  \sum_{\ell=1}^{k} \|  e_h^{n-\ell} \|_h^2.
\end{split}
\end{equation*}
Noting that  $\|e_h^{n}\|_h \le  \|\hat e_h^{n}\|_h$ and rearranging terms, we obtain
\begin{equation*}
\begin{split}
 \frac{\|e_h^{n}\|_h^2 - \|e_h^{n-1}\|_h^2}{\tau}
 \le&  C (h^{2r+2}+ \tau^{2k})   + C \sum_{\ell=1}^{k} \|  e_h^{n-\ell} \|_h^2.
\end{split}
\end{equation*}
Then the discrete Gronwall's inequality implies
\begin{equation*}
\begin{split}
\max_{k \le n\le N}  \|e_h^{n}\|_h^2
 \le&  C (h^{2r+2}+ \tau^{2k})  +  C \sum_{j=0}^{k-1} \|  e_h^{j} \|_h^2.
\end{split}
\end{equation*}
This completes the proof of the theorem.
\end{proof}

\begin{remark}\label{rem:RK}
In Theroem \ref{thm:RK-error}, we discuss the algebraically stable collocation-type method with cut-off technique.
We still prove the optiaml error estimate $O(\tau^k + h^{r+1})$, without the L-stability, i.e. Assumption (P4).
Note that this class of methods includes Gauss--Legendre and Radau IIA methods \cite[Theorem 12.9]{HairerWanner:2010},
while the first one is not L-stable \cite[Table 5.13]{HairerWanner:2010}.
\end{remark}

\section{Fully discrete scheme based on SAV method}\label{sec:sav}
In the preceding section, we develop and analyze a class of maximum bound preserving schemes.
Unfortunately,  the proposed scheme (with relatively large time steps) might produce solutions
with increasing and oscillating energy, see Figure \ref{fig:energy}.
This violates another essential property of the Allen--Cahn model, say energy dissipation.
The aim for this section is to develop
a high-order time stepping schemes via combining the cut-off strategy and
the scalar auxiliary variable (SAV) method.

SAV method is a common-used method for  gradient flow models. It was firstly developed in
\cite{shen2018scalar,shen2019new} and have motived a sequence of interesting work on the
development and analysis of high-order energy-decayed time stepping scheme
in recent years \cite{akrivis2019energy,ShenXu:2018,Gong:2020}.

In particular, assuming that $E_1(u(t)) = \int_{\Omega}F(u(x,t))\d x$  is globally bounded from below,
i.e., $E_1(u(t)) > -C_0$.
we introduce the following scalar auxiliary variable \cite{shen2018scalar}
\begin{equation}\label{eq:scalar}
 z(t) = \sqrt{E_1(u(t)) + C_0}\quad \text{and} \quad W(u) = \dfrac{f(u)}{\sqrt{E_1(u) + C_0}}
\end{equation}
Then the Allen--Cahn equation in \eqref{eqn:AC} can be reformulated as
\begin{equation}\label{eqn:AC-sav}
    \begin{cases}
    u_t = \Delta u + z(t) W(u) &
        \mbox{in~~} \Omega\times(0,T),\\
    u(x,t=0) = u_0(x)   &
        \mbox{in~~} \Omega\times\{0\},\\
    \partial_{\mathbf{n}}u = 0  &
        \mbox{on~~} \partial\Omega\times(0,T)\\
  \end{cases}
\end{equation}
and the scalar auxiliary variable $r(t)$ satisfies
\begin{equation}\label{eqn:r}
\begin{cases}
    z'(t) = -\dfrac12(W(u(t)), u_t(t)),&\qquad \mbox{in~~} (0,T),\\
   z(0) = \sqrt{E_1(u^0) + C_0}.&
 \end{cases}
\end{equation}

One can easily show that the coupled problem \eqref{eqn:AC-sav}-\eqref{eqn:r} is equivalent to the original
equation \eqref{eqn:AC}.  Meanwhile, simple calculation leads to the SAV energy dissipation:
\begin{equation}\label{eq:SAV_energy}
  \frac{\d}{\d t} \Big(\frac12\|\nabla u\|^2 + |z(t)|^2 \Big)=-\|u_t(t)\|^2 \le 0.
\end{equation}
%which indicates the energy dissipation law of Allen--Cahn equation.

Inspired by \cite{akrivis2019energy},
we discretize the coupled problem \eqref{eqn:AC-sav}-\eqref{eqn:r}
by using the  $m$-stage Runge--Kutta method in time (described by Table \ref{tab:RK}) and  lumped mass
finite element method with $r=1$ in space discretization. Then the cut-off operation is applied in each time level
to remove the value violating the maximum bound principle (at nodal points).
For simplicity, we only present the argument for one-dimensional case,
and it can be extended to multi-dimensional cases straightforwardly as mentioned in Remark \ref{rem:highdimension}.

Here we assume that the  $m$-stage Runge--Kutta method (described by Table \ref{tab:RK}) satisfies
the Assumption (P5) and relations \eqref{eqn:accuracy} and \eqref{eqn:stage-order}. Then at $n$-th time level,
with known $u_h^{n-k}, \ldots, u_h^{n-1} \in S_h^r$ and $z^{n-1}\in \mathbb{R}$,
we find $\hat u_h^n \in S_h^r$ and $z^n\in \mathbb{R}$ such that
\begin{equation}\label{eqn:AC-sav-fully}
    \begin{cases}
    \dot u_h^{ni} = \Delta_h u_h^{ni} +  z^{ni} W_h^{ni}&\quad \text{for} ~~i=1,2,\dots,m, \\ %\sum_{\ell=1}^m L_{n-\ell}(t_{n-1}+c_i\tau)\Pi_h W(u_h^{n-j}),\\
    u_h^{ni} = u_h^{n-1} + \tau \sum_{j=1}^m a_{ij} \dot u_h^{nj} &\quad \text{for} ~~i=1,2,\dots,m, \\
    \hat u_h^{n} = u_h^{n-1} + \tau \sum_{i=1}^m b_i \dot u_h^{ni},
  \end{cases}
\end{equation}
and
\begin{equation}\label{eqn:r-fully}
\begin{cases}
    \dot z^{ni} = - \dfrac12(W_h^{ni},   \dot u_h^{ni})_h &\quad \text{for} ~~i=1,2,\dots,m, \\
    z^{ni} = z^{n-1} + \tau \sum_{j=1}^m a_{ij}  \dot z^{nj}&\quad \text{for} ~~i=1,2,\dots,m, \\
   z^{n} = z^{n-1} +  \tau \sum_{i=1}^m b_i  \dot z^{ni},
 \end{cases}
\end{equation}
where $\Pi_h:C(\overline\Omega)\rightarrow S_h^r$ is the Lagrange interpolation operator, and the linearized term $W^{ni}$ is  defined by
$$ W_h^{ni} =  \sum_{\ell=1}^k L_{\ell}(t_{n-1}+c_i\tau)\Pi_h W(u_h^{n-j}),\quad \text{with} ~~k=\min(p,m+1). $$
Then we apply the cut-off operation: find $u_h^n \in S_h^r$ such that
\begin{equation}\label{eqn:truncation-fully-sav}
\begin{split}
u_h^n(x_j)  = \min\big(\max \big(\hat u_h^n(x_j), - \alpha \big) , \alpha \big),  \quad j=0,\dots,Mr.
%u_h^n  = \Pi_h \min\big(\max \big(\hat u_h^n, - \alpha \big) , \alpha \big) .
\end{split}
\end{equation}

\begin{lemma}\label{lem:cutoff-energy}
For $r=1$, the cut-off operation \eqref{eqn:truncation-fully-sav} indicates
\begin{align}
  \| \nabla  u_h^{n}\|_{L^2\II} \le  \| \nabla  \hat u_h^{n}\|_{L^2\II}.\label{grd-decay}
\end{align}
\end{lemma}
\begin{proof}
Since both $\hat u_h^{n}$ and $u_h^{n}$ are piecewise linear, it is easy to see that
$$\| \nabla  u_h^{n}\|_{L^2\II}^2=\frac{1}{h}\sum_{j=1}^M\left|u_h^n(x_j)-u_h^n(x_{j-1})\right|^2,\quad \|\hat u_h^{n}\|_{L^2\II}^2=\frac{1}{h}\sum_{j=1}^M\left|\hat u_h^n(x_j)-\hat u_h^n(x_{j-1})\right|^2.
$$
Obviously, the cut-off operation \eqref{eqn:truncation-fully-sav} derives
$$\left|u_h^n(x_j)-u_h^n(x_{j-1})\right|\le \left|\hat u_h^n(x_j)-\hat u_h^n(x_{j-1})\right|,\quad \text{for}\;j=1,2\cdots,M,$$
which completes the proof.
\end{proof}

The next theorem shows that the cut-off SAV-RK scheme \eqref{eqn:AC-sav-fully}-\eqref{eqn:truncation-fully-sav}
satisfies the energy decay property and discrete maximum bound principle.
\begin{theorem}\label{thm:energy-decay}
Suppose that the Runge--Kutta method in Table \ref{tab:RK} satisfies Assumption (P5),
and we apply the  lumped mass
finite element method with $r=1$ in space discretization.
Then, the time stepping scheme \eqref{eqn:AC-sav-fully}-\eqref{eqn:truncation-fully-sav} satisfies the energy decay property:
\begin{equation}\label{eqn:sav-energy}
\frac12\|  \nabla u_h^n \|_{L^2\II}^2 +  |z^{n}|^2 \le \frac12\|  \nabla u_h^{n-1} \|_{L^2\II}^2 +  |z^{n-1}|^2,\quad \text{for all}~~ n\ge k.
\end{equation}
Meanwhile,  the fully discrete solution \eqref{eqn:AC-sav-fully}-\eqref{eqn:truncation-fully-sav} satisfies the maximum bound principle
\begin{equation}\label{eqn:sav-max}
 \max_{k \le n\le N} | u_h^n (x) | \le \alpha, \quad \text{for all}~~ x\in \Omega.
 \end{equation}
\end{theorem}
\begin{proof}
Due to the cut-off operation in each time level, we know that
$$   \max_{k \le n\le N} | u_h^n (x_j) | \le \alpha, \quad \text{for all}~~ j =0,1,\ldots,M. $$
Since the finite element function is piecewise linear, then for any $x\in (x_{j-1}, x_{j})$
$$   | u_h^n (x) | \le \max\left( | u_h^n (x_{j-1}) |, | u_h^n (x_j) |\right)\le \alpha. $$
Next, we turn to the energy decay property \eqref{eqn:sav-energy}. According to the third relation of \eqref{eqn:AC-sav-fully}, we have
$$\nabla  \hat u_h^{n} = \nabla u_h^{n-1} + \tau \sum_{i=1}^m b_i \nabla \dot u_h^{ni}.$$
Squaring the discrete $L^2$-norms of both sides, yields
$$\| \nabla  \hat u_h^{n}\|^2  = \| \nabla u_h^{n-1}\|^2 + 2\tau \sum_{i=1}^m b_i (\nabla \dot u_h^{ni}, \nabla u_h^{n-1})
+\tau^2\sum_{i,j=1}^m b_ib_j (\nabla \dot u_h^{ni},\nabla \dot u_h^{nj}).$$
By the second relation in \eqref{eqn:AC-sav-fully}, we arrive at
\begin{align*}\| \nabla  \hat u_h^{n}\|^2
&= \| \nabla u_h^{n-1}\|^2 + 2\tau \sum_{i=1}^m b_i (\nabla \dot u_h^{ni}, \nabla u_h^{ni} - \tau \sum_{j=1}^m a_{ij} \nabla \dot u_h^{ni})
+\tau^2\sum_{i,j=1}^m b_ib_j (\nabla \dot u_h^{ni},\nabla \dot u_h^{nj})\\
&= \| \nabla u_h^{n-1}\|^2 + 2\tau \sum_{i=1}^m b_i (\nabla \dot u_h^{ni}, \nabla u_h^{ni})
 - \tau^2\sum_{i,j=1}^m m_{ij} (\nabla \dot u_h^{ni},\nabla \dot u_h^{nj})\\
 &\le \| \nabla u_h^{n-1}\|^2 + 2\tau \sum_{i=1}^m b_i (\nabla \dot u_h^{ni}, \nabla u_h^{ni}),
\end{align*}
where we apply the Assumption (P4) in the last inequality. Then we apply the first relation in \eqref{eqn:AC-sav-fully} to derive
\begin{align*}
\| \nabla  \hat u_h^{n}\|^2
&= \| \nabla u_h^{n-1}\|^2 - 2\tau \sum_{i=1}^m b_i \| \dot u_h^{ni} \|^2 + 2\tau \sum_{i=1}^m b_i z^{ni}  (  \dot u_h^{ni}, W_h^{ni}  )_h
\end{align*}
On the other hand, the similar argument also leads to
\begin{align*}
|z^{n}|^2
&\le |z^{n-1}|^2- \tau \sum_{i=1}^m b_i   z^{ni}  (  \dot u_h^{ni}, W_h^{ni}  )_h\end{align*}
Therefore we conclude that
\begin{align*}%label{eqn:decay-1}
\frac12 \| \nabla  \hat u_h^{n}\|_h^2  + |z^n|^2
&\le \frac12 \| \nabla u_h^{n-1}\|_h^2  + |z^{n-1}|^2
- \tau \sum_{i=1}^m b_i \| \dot u_h^{ni} \|_h^2 \le \frac12 \| \nabla u_h^{n-1}\|_h^2  + |z^{n-1}|^2  . %- 2\tau \sum_{i=1}^m b_i \dot r^{ni}  (  \dot u_h^{ni}, W_h^{ni}  )_h
\end{align*}
%Note that for $r=1$, the cut-off operation \eqref{eqn:truncation-fully-sav} indicates
%\begin{align*}
%  \| \nabla  u_h^{n}\|_{L^2\II} \le  \| \nabla  \hat u_h^{n}\|_{L^2\II},
%\end{align*}
which together with \eqref{grd-decay} implies the desired energy decay property immediately.
\end{proof}

\begin{remark}
Note that the energy dissipation law holds valid only if $r=1$,
since in this case the cut-off operation does not enlarge the $H^1$ semi-norm, which is present as \eqref{grd-decay} in Lemma \ref{lem:cutoff-energy}. This property is not clear for finite element method with high degree polynomials. Hence,
how to design an spatially high-order (unconditionally) energy dissipative and maximum bound preserving scheme
is still unclear and warrants further investigation.
\end{remark}

Next, we shall derive an error estimate for the fully discrete scheme \eqref{eqn:AC-sav-fully}-\eqref{eqn:truncation-fully-sav}.
To begin with, we shall examine the local truncation error. We define the local truncation error $\eta_{ni}$ and $\eta_{n}$  as
\begin{equation}\label{eqn:AC-sav-fully-trunc}
    \begin{cases}
    \dot u_*^{ni} = \Delta  u(t_{ni}) +  z(t_{ni}) W_*^{ni}&\quad \text{for} ~~i=1,2,\dots,m, \\ %\sum_{\ell=1}^m L_{n-\ell}(t_{n-1}+c_i\tau)\Pi_h W(u_h^{n-j}),\\
    u(t_{ni})=  u(t_{n-1}) + \tau \sum_{j=1}^m a_{ij} \dot u_*^{nj} + \eta_{ni}&\quad \text{for} ~~i=1,2,\dots,m, \\
    u(t_n) = u(t_{n-1}) + \tau \sum_{i=1}^m b_i \dot u_*^{ni} + \eta_{n}
  \end{cases}
\end{equation}
where $t_{ni}=t_{n-1}+c_i\tau$ and $W_*^{ni}$ denotes the extrapolation
$$ W_*^{ni} =  \sum_{\ell=1}^m L_{\ell}(t_{n-1}+c_i\tau) W(u(t_{n-j})). $$
Similarly, we define $d_{ni}$ and $d_{n}$ as
\begin{equation}\label{eqn:r-fully-trunc}
\begin{cases}
    \dot z_*^{ni} = - \dfrac12(W_*^{ni},   \dot u^{ni}_*) &\quad \text{for} ~~i=1,2,\dots,m, \\
    z(t_{ni}) = z(t_{n-1}) + \tau \sum_{j=1}^m a_{ij}  \dot z_*^{nj}+ d_{ni}&\quad \text{for} ~~i=1,2,\dots,m, \\
   z(t_{n}) = z(t_{n-1}) +  \tau \sum_{i=1}^m b_i  \dot z_*^{ni}+d_{n},
 \end{cases}
\end{equation}
Provided the assumption (P5) and relations \eqref{eqn:accuracy} and \eqref{eqn:stage-order},
the local truncation errors $\eta_{ni}$,   $\eta_{n}$,  $d_{ni}$, $d_n$ satisfy the estimate
\begin{equation}\label{eqn:trunc}
\| \eta_n \|_{H^1\II} + | d_n | + \tau \sum_{i=1}^m \Big(\| \eta_{ni} \|_{H^1\II} + |d_{ni}|\Big) \le C \tau^{k+1}.
\end{equation}
We omit the proof, since it is similar to the one of Lemma \ref{lem:RK-local-truncation}, given in Appendix.
See also \cite[Lemma 3.1]{akrivis2019energy}.

%{\color{red}We need a proof of \eqref{eqn:trunc}, also express $C$ explicitly in terms of $u$ and $f(u)$.}
%By  \eqref{eqn:AC-sav-fully-trunc} and \eqref{eqn:r-fully-trunc}, we derive that
%\begin{equation}\label{eqn:AC-sav-fully-trunc-2}
%    \begin{cases}
%   \Pi_h \dot u_*^{ni} =  \Delta_h \Pi_h  u(t_{ni})+  z(t_{ni}) \Pi_h W_*^{ni} + (\Pi_h \Delta - \Delta_h \Pi_h)  u(t_{ni})
%   &\quad \text{for} ~~i=1,2,\dots,m, \\
%   \Pi_h  u(t_{ni})=  u(t_{n-1}) + \tau \sum_{j=1}^m a_{ij} \Pi_h \dot u_*^{ni} + \Pi_h \eta_{ni}&\quad \text{for} ~~i=1,2,\dots,m, \\
%  \Pi_h  u(t_n) = \Pi_h u(t_{n-1}) + \tau \sum_{i=1}^m b_i \Pi_h \dot u_*^{ni} + \Pi_h \eta_{n}
%  \end{cases}
%\end{equation}
%and
%\begin{equation}\label{eqn:r-fully-trunc}
%\begin{cases}
%    \dot z_*^{ni} = \dfrac12(W_*^{ni} ,   \dot u^{ni})_h + \dfrac12\Big((W_*^{ni} ,   \dot u^{ni})- (W_*^{ni} ,   \dot u^{ni})_h\Big) &\quad \text{for} ~~i=1,2,\dots,m, \\
%    z(t_{ni}) = r(t_{n-1}) + \tau \sum a_{ij}  \dot z_*^{nj}+ d_{ni}&\quad \text{for} ~~i=1,2,\dots,m, \\
%   z(t_{n}) = r(t_{n-1}) +  \tau \sum b_i  \dot z_*^{ni}+d_{n},
% \end{cases}
%\end{equation}

\begin{theorem}\label{thm:sav-error}
Suppose that the Runge--Kutta method satisfies Assumption (P4) and the relations \eqref{eqn:accuracy} and \eqref{eqn:stage-order}.
%and we apply the  lumped mass finite element method with $r=1$ in space discretization.
Assume that $|u_0| \le \alpha$ and the maximum principle \eqref{eqn:max-AC} holds,
and assume that the starting values $u_h^l$ and $z^l$, $l=0,\dots,k-1$, are given and
$$
|u_h^l(x_j)|\le \alpha , \quad j=0,\dots,M,\quad l=0,\dots,k-1.
$$
Then the   {fully discrete} solution given by \eqref{eqn:AC-sav-fully}-\eqref{eqn:truncation-fully-sav} satisfies
%\begin{align}\label{AC-fully-discrete-max-pr}
%|u_h^n(x_j)|\le \alpha , \quad j=0,\dots,Mr,\quad n=k,\dots,N,
%\end{align}
%and
{for $n=k,\ldots,N$}
\begin{align}\label{fully-err-AC}
  \| u(t_n) - u_h^n \|_{L^2(\Omega)} \le  C(\tau^k + h^{2}) + C\sum_{l=0}^{k-1}
  \| u(t_l) - u_h^l \|_{L^2(\Omega)} + C|z(t_{k-1}) - z^{k-1}| ,
\end{align}
provided that $u, f$ and $f(u)$ are sufficiently smooth in both time and space variables.
%\in C^{k+1}([0,T];C(\bar \Omega)) \cap C^{k}([0,T];\text{Dom}(\Delta)) \cap C^1([0,T];H^{4}(\Omega))$,
%$f$ is locally Lipschitz continuous and $f(u)\in C^k([0,T];L^2\II)\cap C([0,T];H^{4}\II)$.
\end{theorem}

\begin{proof}
Subtracting \eqref{eqn:AC-sav-fully}-\eqref{eqn:r-fully} from \eqref{eqn:AC-sav-fully-trunc}-\eqref{eqn:r-fully-trunc},
and with the notation
\begin{align*}
  &e_h^{ni} = \Pi_h u(t_{ni}) - u_h^{ni},\quad  &&\dot e_h^{ni} = \Pi_h \dot u_*^{ni} - \dot u_h^{ni},
 &&&  e_h^{n} = \Pi_h u(t_n) - u_h^{n},  \quad  \hat e_h^{n} = \Pi_h u(t_n) - \hat u_h^{n},  \\
 & \xi^{ni} = z(t_{ni}) - z^{ni}, \quad  && \dot \xi^{ni} = \dot z_*^{ni} - \dot z^{ni} ,  &&&  \xi^{n} = z(t_n)- z^n\quad  .
\end{align*}
 we have the following error equations
\begin{equation}\label{eqn:AC-sav-fully-err}
    \begin{cases}
   \dot e_h^{ni} =  \Delta_h e_h^{ni} +  (z(t_{ni}) \Pi_h W_*^{ni} - z^{ni} W_h^{ni}) + (\Pi_h \Delta - \Delta_h \Pi_h) u(t_{n-1})
   &\quad \text{for} ~~i=1,2,\dots,m, \\
   e_h^{ni}=  e_h^{n-1}  + \tau \sum_{j=1}^m a_{ij} \dot e^{nj} + \Pi_h \eta_{ni}&\quad \text{for} ~~i=1,2,\dots,m, \\
 \hat e_h^{n} = e_h^{n-1} + \tau \sum_{i=1}^m b_i \dot e_h^{ni} + \Pi_h \eta_{n}
  \end{cases}
\end{equation}
and
\begin{equation}\label{eqn:r-fully-err}
\begin{cases}
    \dot \xi^{ni} = - \dfrac12(W_*^{ni} ,   \dot u_*^{ni}) + \dfrac12(W_h^{ni} ,   \dot u_h^{ni})_h &\quad \text{for} ~~i=1,2,\dots,m, \\
    \xi^{ni} = \xi^{n-1} + \tau \sum_{j=1}^m a_{ij}  \dot \xi^{nj}+ d_{ni}&\quad \text{for} ~~i=1,2,\dots,m, \\
   \xi^n =\xi^{n-1}  +  \tau \sum_{j=1}^m  b_i  \dot \xi^{ni}+d_{n},
 \end{cases}
\end{equation}

Now, take the square of discrete $L^2$ norm of both sides
of the last relation of equation \eqref{eqn:AC-sav-fully-err},
we can get
\begin{equation}\label{eqn:split-0}
\begin{split}
\|\hat e_h^{n}\|_h^2 & = \|e_h^{n-1}+\tau\sum_{i=1}^{m} b_i\dot{e}_h^{ni}\|_h^2
+ 2(\eta^{n}, e_h^{n-1}+ \tau\sum_{i=1}^{m} b_i\dot{e}_h^{ni})_h + \|\Pi_h \eta^{n}\|_h^2.
\end{split}
\end{equation}
For the first term on the right hand side,
we expand it and apply the second equation of \eqref{eqn:AC-sav-fully-err}  to obtain
\begin{align*}
\|e_h^{n-1}+\tau\sum_{i=1}^{m} b_i\dot{e}_h^{ni}\|_h^2
&= \|e_h^{n-1}\|_h^2 + 2\tau\sum_{i=1}^{m}b_i(\dot{e}_h^{ni},e_h^{ni}-\eta_{ni})_h
- \tau^2\sum_{i,j=1}^{m}m_{ij}(\dot{e}_h^{ni}, \dot{e}_h^{nj})_h\\
&\leq \|e_h^{n-1}\|_h^2 + 2\tau\sum_{i=1}^{m}b_i(\dot{e}_{ni},e_h^{ni}-\eta_{ni})_h,
\end{align*}
where in the last inequality we use the positive semi-definiteness of the matrix $\mathcal{M}$ in Assumption (P4).
Next, we note that the relation of \eqref{eqn:AC-sav-fully-err} implies
\begin{equation*}
\begin{split}
(\dot{e}_h^{ni},e_h^{ni}-\eta_{ni})_h
&= \Big(   \Delta_h e_h^{ni} +  (z(t_{ni}) \Pi_h W_*^{ni} - z^{ni} W_h^{ni})
+ (\Pi_h \Delta - \Delta_h \Pi_h) u(t_{n-1}), e_h^{ni}-\eta_{ni} \Big)_h\\
&= -\|\nabla e_h^{ni}\|_{L^2\II}^2 + (\nabla e_h^{ni} , \nabla \Pi_h \eta_{ni})
+  \Big(z(t_{ni}) \Pi_h W_*^{ni} - z^{ni} W_h^{ni}, e_h^{ni}-\eta_{ni}\Big)_h  \\
& \quad +  \Big(  (\Pi_h \Delta - \Delta_h \Pi_h) u(t_{n-1}), e_h^{ni}-\eta_{ni}\Big)_h
\end{split}
\end{equation*}
The bound of second term of the right hand side can be derived via Cauchy-Schwarz inequality
\begin{equation*}
|(\nabla e_h^{ni} , \nabla \Pi_h \eta_{ni})| \le \frac14 \|  \nabla e_h^{ni} \|_{L^2\II}^2 + C \| \eta_{ni}\|_{H^1\II}^2.
\end{equation*}
Then the third term can be bounded as
\begin{equation*}
\begin{split}
  \Big(z(t_{ni}) \Pi_h W_*^{ni} - z^{ni} W_h^{ni}, e_h^{ni}-\eta_{ni} \Big)_h
 & \le z(t_{ni})  \Big( \Pi_h W_*^{ni} -  W_h^{ni}, e_h^{ni}-\eta_{ni} \Big)_h +  \xi^{ni} \Big( W_h^{ni}, e_h^{ni}-\eta_{ni} \Big)_h\\
  &\le C \Big(\sum_{j=1}^k \| e_h^{n-j} \|_h^2 + \|  e_h^{ni} \|_h^2 + \| \Pi_h \eta_{ni} \|_{L^2\II}^2 + |\xi^{ni}|^2 \Big).
\end{split}
\end{equation*}
The bound of the last term follows from Lemma \ref{lem:truncation}
\begin{equation*}
\begin{split}
 \Big(  (\Pi_h \Delta - \Delta_h \Pi_h) u(t_{n-1}), e_h^{ni}-\eta_{ni} \Big)_h &\le C h^2   \| e_h^{ni}-\eta_{ni}   \|_{H^1\II} \\
 &\le  \frac14\| \nabla e_h^{ni}  \|_{L^2\II}^2 + C(\|  e_h^{ni} \|_h^2 + \|  \eta_{ni} \|_{H^1\II}^2+ h^4 )
\end{split}
\end{equation*}
Therefore, we arrive at
\begin{equation*}
2(\dot{e}_h^{ni},e_h^{ni}-\eta_{ni})_h \le  -\|\nabla e_h^{ni}\|_{L^2\II}^2 +
 C \Big(  \sum_{j=1}^k \| e_h^{n-j} \|_h^2 + \|  e_h^{ni} \|_h^2  + |\xi^{ni}|^2 + \| \eta_{ni}\|_{H^1\II}^2 + h^2  \Big),
\end{equation*}
and hence
\begin{align*}
\|e_h^{n-1}+\tau\sum_{i=1}^{m} b_i\dot{e}_h^{ni}\|_h^2
&\leq \|e_h^{n-1}\|_h^2 - \tau\sum_{i=1}^{m}b_i \| \nabla e_h^{ni}  \|_{L^2\II}^2 + C\tau \sum_{i=1}^m \Big( |\xi^{ni}|^2 + \|  e_h^{ni} \|_h^2\Big) \\
&\qquad + C\tau \sum_{j=1}^k \| e_h^{n-j} \|_h^2 + C \tau(h^4+\tau^{2k}).
\end{align*}
In view of the first relation of the error equation \eqref{eqn:AC-sav-fully-err}, we have the estimate
\begin{align*}
(\eta^{n}, e_h^{n-1}+ \tau\sum_{i=1}^{m} b_i\dot{e}_h^{ni})_h
&\le \| \eta_n  \|_h  \|e_h^{n-1}\|_h + C \tau \|  \eta_n \|_{H^1\II} \sum_{i=1}^m b_i
\Big(\| \nabla e_h^{ni} \|_h + \sum_{j=1}^{k} \|  e_h^{n-j} \|_h + |\xi^{ni}| + h^2\Big)\\
&\le C\tau(h^4+ \tau^{2k})   + \frac{\tau}{4}   \sum_{i=1}^m b_i \Big(\| \nabla e_h^{ni} \|_h^2+ |\xi^{ni}|^2 \Big)
+ C\tau  \sum_{j=1}^{k} \|  e_h^{n-j} \|_h^2
\end{align*}
which gives a bound of the second term in \eqref{eqn:split-0}. In conclusion, we obtain that
\begin{equation}\label{eqn:sav-est-01}
\begin{split}
 \|\hat e_h^{n}\|_h^2 + \frac\tau2 \sum_{i=1}^m \| \nabla e_h^{ni} \|_{L^2\II}^2& \le  C\tau(h^4+ \tau^{2k}) +  \|e_h^{n-1}\|_h^2  \\
&\quad + C \tau   \sum_{i=1}^m \Big(\|  e_h^{ni} \|_h^2+ |\xi^{ni}|^2 \Big)
+ C\tau  \sum_{j=1}^{k} \|  e_h^{n-j} \|_h^2.
\end{split}
\end{equation}
Similarly, from \eqref{eqn:r-fully-err} and \eqref{eqn:trunc} we can derive
\begin{equation*}
\begin{split}
|\xi^n|^2 &\le C\tau(h^4+ \tau^{2k}) + (1+c \tau) |\xi^{n-1}|^2 + \frac \tau4 \sum_{i=1}^m \| \nabla e_h^{ni} \|_{L^2\II}^2 \\
&\quad + C \tau   \sum_{i=1}^m \Big(\|  e_h^{ni} \|_h^2+ |\xi^{ni}|^2 \Big)
+ C\tau  \sum_{j=1}^{k} \|  e_h^{n-j} \|_h^2,
\end{split}
\end{equation*}
where we use the estimate that
\begin{align*}
  (W_*^{ni} ,   \dot u_*^{ni}) -  (W_h^{ni} ,   \dot u_h^{ni})_h &=  (W_*^{ni} ,   \dot u_*^{ni})  -  (W_*^{ni} ,   \dot u_*^{ni}) _h
   + (W_*^{ni}-W_h^{ni},  \dot u_*^{ni})_h + (W_h^{ni},  \dot e_h^{ni} )_h \\
  & \le C h^{2} + C\sum_{j=1}^k \|  e_h^{n-j} \|_h \|  \Pi_h \dot u_*^{ni} \|_h + (\nabla W_h^{ni},  \nabla  e_h^{ni} )_h\\
 &+ ( W_h^{ni}, z(t_{ni}) \Pi_h W_*^{ni} - z^{ni} W_h^{ni} )_h
  +  ( W_h^{ni}, (\Pi_h \Delta - \Delta_h \Pi_h) u(t_{n-1}) )_h \\
% [\text{use}~~ \| \nabla u_h^{n} \|\le C ~~ \text{by Theorem \ref{thm:energy-decay}}] \quad
 & \le C h^{2} + C\sum_{j=1}^k \|  e_h^{n-j} \|_h + C \| \nabla  e_h^{ni} \|  + C|\xi^{ni}|,
\end{align*}
%as well as the fact that $\|\nabla u_h^n\|_{L^2\II}$ are uniformly bounded for all $n$.
where we use the fact that $ \| \nabla u_h^{n} \|\le C$ (by Theorem \ref{thm:energy-decay}) in the last inequality. To sum up, we arrive at
\begin{equation*}
\begin{split}
 \|\hat e_h^{n}\|_h^2 + |\xi^n|^2 + \frac\tau4 \sum_{i=1}^m \| \nabla e_h^{ni} \|_{L^2\II}^2
 \le&  C\tau(h^4+ \tau^{2k}) +  \|e_h^{n-1}\|_h^2 + (1+c \tau) |\xi^{n-1}|^2   \\
&\quad + C \tau   \sum_{i=1}^m \Big(\|  e_h^{ni} \|_h^2+ |\xi^{ni}|^2 \Big)
+ C\tau  \sum_{j=1}^{k} \|  e_h^{n-j} \|_h^2.
\end{split}
\end{equation*}
Note that $| e_h^{n}(x_j)| \le |\hat e_h^{n}(x_j)|$ for all $j=0,1,\ldots,M$, which implies
\begin{equation}\label{eqn:sav-est-02}
\begin{split}
 \|e_h^{n}\|_h^2 + |\xi^n|^2 + \frac\tau4 \sum_{i=1}^m \| \nabla e_h^{ni} \|_{L^2\II}^2
 \le&  C\tau(h^4+ \tau^{2k}) +  \|e_h^{n-1}\|_h^2 + (1+c \tau) |\xi^{n-1}|^2   \\
&\quad + C \tau   \sum_{i=1}^m \Big(\|  e_h^{ni} \|_h^2+ |\xi^{ni}|^2 \Big)
+ C\tau  \sum_{j=1}^{k} \|  e_h^{n-j} \|_h^2.
\end{split}
\end{equation}

Next, we shall derive a bound for  $ \sum_{i=1}^m \Big(\|  e_h^{ni} \|_h^2+ |\xi^{ni}|^2 \Big)$
 on the right-hand side. To this
end, we test the second relation of \eqref{eqn:AC-sav-fully-err} by $e_h^{ni}$. This yields
\begin{equation*}
\begin{split}
 \sum_{i=1}^m \|e_h^{ni}\|_h^2 & \le C \|  e_h^{n-1} \|_h^2 + C\tau\sum_{i,j=1}^m a_{ij}(\dot e_h^{nj},e_h^{ni}) + C\sum_{i=1}^m \| \Pi_h \eta_{ni} \|_h^2\\
& \le  C \|  e_h^{n-1} \|_h^2 + C\tau\sum_{i,j=1}^m a_{ij}(\dot e_h^{nj},e_h^{ni}) _h + C\tau^{2k}.
\end{split}
\end{equation*}
Then, we apply the first relation of  \eqref{eqn:AC-sav-fully-err} and Lemma \ref{lem:truncation} to derive
\begin{equation*}
\begin{split}
 \sum_{i,j=1}^m a_{ij}(\dot e_h^{nj},e_h^{ni})_h &= - \sum_{i,j=1}^m a_{ij}(\nabla e_h^{nj}, \nabla e_h^{ni})
 + \sum_{i,j=1}^m a_{ij}   (z(t_{ni}) \Pi_h W_*^{ni} - z^{ni} W_h^{ni}, e_h^{ni})_h\\
&\quad + \sum_{i,j=1}^m a_{ij}  ((\Pi_h \Delta - \Delta_h \Pi_h) u(t_{n-1}),  e_h^{ni} )_h \\
&\le C\sum_{i=1}^m \Big(\| \nabla e_h^{ni}  \|_{L^2\II}^2 + \|  e_h^{ni}  \|_{h}^2 + |\xi^{ni}|^2 \Big)+ Ch^4 + C\sum_{j=1}^k \| e_h^{n-j} \|_h^2.
\end{split}
\end{equation*}
Therefore, we obtain
\begin{equation*}
\begin{split}
 \sum_{i=1}^m \|e_h^{ni}\|_h^2   \le  C( \tau h^4+\tau^{2k}) + C\|  e_h^{n-1} \|_h^2 + C\tau \sum_{j=1}^k \| e_h^{n-j} \|_h^2 +
 C\tau \sum_{i=1}^m \Big(\| \nabla e_h^{ni}  \|_{L^2\II}^2 + \|  e_h^{ni}  \|_{h}^2 + |\xi^{ni}|^2 \Big).
\end{split}
\end{equation*}
Similarly, from \eqref{eqn:r-fully-err} we can derive
\begin{equation*}
\begin{split}
 \sum_{i=1}^m |\xi^{ni}|^2 & \le C |\xi^{n-1}|^2 + C\tau\sum_{i,j=1}^m a_{ij} \dot \xi^{nj} \xi^{ni} + C\sum_{i=1}^m |d_{ni}|^2\\
%& \le  C |\xi^{n-1}|^2 + C\tau\sum_{i,j=1}^m a_{ij} \dot \xi^{nj} \xi^{ni} + C\tau^{2k} \\
&\le C(\tau h^4 + \tau^{2k})+C|\xi^{n-1}|^2  + C\tau \sum_{j=1}^k \| e_h^{n-j} \|_h^2 +
 C\tau \sum_{i=1}^m \Big(\| \nabla e_h^{ni}  \|_{L^2\II}^2 + \|  e_h^{ni}  \|_{h}^2 + |\xi^{ni}|^2 \Big)
\end{split}
\end{equation*}
Sum up these two estimates and note that,
for sufficiently small $\tau$,
\begin{equation*}
\begin{split}
 \sum_{i=1}^m \Big(\|e_h^{ni}\|_h^2+ |\xi^{ni}|^2\Big)  \le C(\tau h^4 + \tau^{2k})+C|\xi^{n-1}|^2
 + C\tau \sum_{j=1}^k \| e_h^{n-j} \|_h^2 +
 C\tau \sum_{i=1}^m \| \nabla e_h^{ni}  \|_{L^2\II}^2.
\end{split}
\end{equation*}
Now substituting the above estimate into \eqref{eqn:sav-est-02}, we have
\begin{equation*}%\label{eqn:sav-est-07}
\begin{split}
 \|e_h^{n}\|_h^2 + |\xi^n|^2 + \frac\tau4 \sum_{i=1}^m \| \nabla e_h^{ni} \|_{L^2\II}^2
 \le&  C\tau(h^4+ \tau^{2k}) +  \|e_h^{n-1}\|_h^2 + (1+C \tau) |\xi^{n-1}|^2   \\
&\quad + C\tau^2 \sum_{i=1}^m \| \nabla e_h^{ni}  \|_{L^2\II}^2
+ C\tau  \sum_{j=1}^{k} \|  e_h^{n-j} \|_h^2.
\end{split}
\end{equation*}
Then for sufficiently small $\tau$, there holds
\begin{equation*}
\begin{split}
 \|e_h^{n}\|_h^2 + |\xi^n|^2
 \le&  C\tau(h^4+ \tau^{2k}) +  \|e_h^{n-1}\|_h^2 + (1+C \tau) |\xi^{n-1}|^2   + C\tau  \sum_{j=1}^{k} \|  e_h^{n-j} \|_h^2.
\end{split}
\end{equation*}
Rearranging terms, we obtain
\begin{equation*}
\begin{split}
 \frac{(\|e_h^{n}\|_h^2 + |\xi^n|^2) - (\|e_h^{n-1}\|_h^2 + |\xi^{n-1}|^2)}{\tau}
 \le&  C (h^4+ \tau^{2k})  + C |\xi^{n-1}|^2   + C \sum_{j=1}^{k} \|  e_h^{n-j} \|_h^2.
\end{split}
\end{equation*}
Then the discrete Gronwall's inequality implies
\begin{equation*}
\begin{split}
\max_{k \le n\le N} \Big(\|e_h^{n}\|_h^2 + |\xi^n|^2\Big)
 \le&  C (h^4+ \tau^{2k})  + C |\xi^{k-1}|^2    + C \sum_{j=0}^{k-1} \|  e_h^{j} \|_h^2.
\end{split}
\end{equation*}
This completes the proof of the theorem.
\end{proof}

\section{Numerical Results}\label{sec:numerics}
In this section, we present numerical results to illustrate the the theoretical results
with a one-dimensional example:
\begin{equation}\label{eq:num1}
\begin{cases}
  \partial_t u = \partial_{xx}u + f(u), & \mbox{in } \Omega \times(0,T], \\
  \partial_x u = 0, & \mbox{on } \partial\Omega \times(0,T] \\
  u(x,t=0) = u_0(x) & \mbox{in } \Omega,
\end{cases}
\end{equation}
where  $\Omega = (0,2)$ and $f(u) = \epsilon^{-2}(u-u^3)$ with $\epsilon=0.1$ is
the Ginzburg-Landau double-well potential.
The initial value satisfies the maximum principle given by
\begin{equation}\label{eq:num1_init}
u_0(x) = \begin{cases}
           1, & \mbox{if } 0<x<1/2, \\
           \cos\l(\frac23\pi\l(x+\frac12\r)\r), & \mbox{if }  1/2\leq x<2.
         \end{cases}
\end{equation}
The smooth initial value is chosen to satisfy the Neumann boundary condition.

We solve the problem \eqref{eq:num1} with spatial mesh size $h = 2/N_x$
and temporal mesh size $\tau = T / N_t$, with $T=\epsilon^2$ and $5\epsilon^2$.
Throughout the section, we shall apply the Gauss--Legendre methods with $m=1,2,3$ and hence $k=2,3,4$.
We compute the numerical solution at the first $k -1$ time levels by using the three-stage
Gauss--Legendre Runge--Kutta method \cite[Table 5.2]{HairerWanner:2010}, which has sixth-order accuracy in time.
Cutting off the numerical solutions at the first $k - 1$ time levels does not affect the global accuracy.

Since the closed form of exact solution is
unavailable, we compare our numerical solution with a reference solution computed by a
high-order method (i.e. cut-off RK method with $r=3$, $m=3$) with small mesh sizes.
In particular, the temporal error $e_\tau$
is computed by fixing the spatial mesh size $h = 2/400$
and comparing the numerical solution with
a reference solution (with $\tau = T/1000$).
Similarly, the spatial error $e_h$
is computed to by fixing the temporal step size
$\tau = T/1000$ and comparing the numerical solutions with
a reference solution (with  $h = 2/400$).

In Table \ref{table:ex1_space}, we present the spatial errors of both
cut-off RK schemes \eqref{eqn:RK-fully}-\eqref{eqn:RK-cut} with $r=1,2,3$ and the cut-off SAV-RK scheme
\eqref{eqn:AC-sav-fully}-\eqref{eqn:truncation-fully-sav}  with $r=1$.
Numerical results show the optimal rate $O(h^{r+1})$, which fully supports our theoretical results in
Theorems \ref{thm:RK-error} and \ref{thm:sav-error}.
Temporal errors are presented in \ref{table:ex1_time} and \ref{table:ex1_time-sav},
both of which show the empirical convergence rate $O(\tau^{m+1})$
and hence coincidence to Theorems \ref{thm:RK-error} and \ref{thm:sav-error}.

\begin{table}[h]
\centering
\caption{$e_h$ of cut-off RK  \eqref{eqn:RK-fully}-\eqref{eqn:RK-cut}  and cut-off SAV-RK
\eqref{eqn:AC-sav-fully}-\eqref{eqn:truncation-fully-sav}. }
\begin{tabular}{|c|c|ccccc|c|}
\hline
$r\backslash N_x$  & $T$   &10 &20 &40 &80 &160&   rate \\
\hline
RK   &   $0.01$& 3.03e-2& 7.42e-3& 1.84e-3& 4.60e-4& 1.14e-4& $\approx$ 2.00 (2.00)    \\
 (r=1)   &   $0.05$& 1.49e-1& 1.03e-2& 2.32e-3& 5.71e-4& 1.43e-4& $\approx$ 2.01 (2.00)    \\
\hline
RK  &   $0.01$& 4.37e-3& 4.99e-4& 5.90e-5& 7.27e-6& 9.05e-7& $\approx$ 3.01 (3.00)    \\
(r=2)    &   $0.05$& 6.15e-2& 1.64e-3& 1.73e-4& 2.09e-5& 2.60e-6& $\approx$ 3.03 (3.00)    \\
\hline
RK   &   $0.01$& 5.10e-4& 3.19e-5& 1.99e-6& 1.23e-7& 7.74e-9&$\approx$ 4.00 (4.00)   \\
 (r=3)   &   $0.05$& 5.89e-3& 1.21e-4& 8.12e-6& 5.03e-7& 3.14e-8&$\approx$ 4.01 (4.00)    \\
\hline
%\rowcolor{cyan}
SAV-RK &$0.01$& 3.03e-2& 7.42e-3& 1.84e-2& 4.62e-4& 1.17e-4& $\approx$ 2.00 (2.00)\\
%\rowcolor{cyan}
(r=1)  &$0.05$& 1.49e-1& 1.03e-2& 2.34e-3& 5.85e-4& 1.56e-4& $\approx$ 2.01 (2.00)\\
\hline
\end{tabular}
\label{table:ex1_space}
\end{table}

\begin{table}[h]
\centering
\caption{$e_\tau$ of cut-off RK scheme \eqref{eqn:RK-fully}-\eqref{eqn:RK-cut}, with $\tau=T/N_t$.}
\label{table:ex1_time}
\begin{tabular}{|c|c|cccccc|c|}
\hline
$m\backslash N_t$  & $T$   &10 &20 &40 &80 &160 & 320 & rate \\
\hline
1   &   $ 0.01$& 3.76e-4& 9.61e-5& 2.43e-5& 6.10e-5& 1.53e-6& 3.82e-7 &$\approx$ 1.99 (2.00)    \\
    &   $ 0.05$ & 8.01e-4& 5.36e-5& 1.16e-5& 2.71e-6& 6.56e-7& 1.61e-7 &$\approx$ 2.06(2.00)   \\
\hline
2   &   $ 0.01$& 4.92e-5& 6.20e-6& 7.74e-7& 9.65e-8& 1.21e-8& 1.51e-9 &$\approx$ 3.00 (3.00)    \\
    &   $ 0.05$ & 1.73e-2& 3.60e-5& 1.78e-6& 2.08e-7& 2.51e-8& 3.08e-9 &$\approx$ 3.06 (3.00)    \\
\hline
3   &   $0.01$& 1.05e-5& 6.83e-7& 4.31e-8& 2.71e-9& 1.69e-10& 1.05e-11 &$\approx$ 4.00 (4.00)    \\
    &   $0.05$ & 2.88e-2& 3.66e-3& 3.82e-7& 1.56e-8& 9.61e-10& 6.06e-11 &$\approx$ 4.21 (4.00)    \\
\hline
\end{tabular}
\end{table}

 \begin{table}[!h]
\centering
\caption{$e_\tau$ of cut-off SAV-RK scheme \eqref{eqn:AC-sav-fully}-\eqref{eqn:truncation-fully-sav}, with $\tau=T/N_t$.}
\label{table:ex1_time-sav}
\begin{tabular}{|c|c|cccccc|c|}
\hline
$m\backslash N_t$  & $T$   &10 &20 &40 &80 &160 & 320 & rate \\
\hline
1   &   $0.01$& 8.08e-3& 2.23e-3& 5.96e-4& 1.53e-4& 3.79e-5& 8.78e-6 &$\approx$ 2.03 (2.00)    \\
    &   $0.05$& 7.94e-4& 1.79e-4& 4.80e-5& 1.24e-5& 3.09e-6& 7.17e-7 &$\approx$ 2.00 (2.00)   \\
\hline
2   &   $0.01$& 5.56e-9& 5.95e-4& 8.82e-5& 1.11e-5& 1.37e-6& 1.65e-7 &$\approx$ 3.02 (3.00)    \\
    &   $0.05$& 1.47e-2& 5.17e-5& 7.17e-6& 1.00e-6& 1.31e-7& 1.63e-8 &$\approx$ 2.97 (3.00)    \\
\hline
3   &   $0.01$& 6.97e-11& 2.56e-4& 2.47e-5& 1.66e-6& 1.06e-7& 6.60e-9 &$\approx$ 3.95 (4.00)    \\
    &   $0.05$& 2.45e-2& 2.86e-3& 7.73e-7& 6.16e-8& 4.38e-9& 2.93e-10 &$\approx$ 3.79 (4.00)    \\
\hline
\end{tabular}
\end{table}

In Figure 4.1, we plot
 the maximal cut-off value at each step
$$\rho^n = \max_{0\le j\le Mr+1} |u_h^n(x_j) -\hat u_h^n(x_j)|$$
and the error of the numerical solution $e(x) = u^N_h (x) - u(x, T)$.
Our numerical results show that the cut-off operation is active in the computation.
Meanwhile, we observe that a coarse step mesh will result in a larger cut-off value,
without affecting the convergence rate.

\begin{figure}[h]
\centering
\includegraphics[trim = .1cm .1cm .1cm .1cm, clip=true,width=0.75\textwidth]{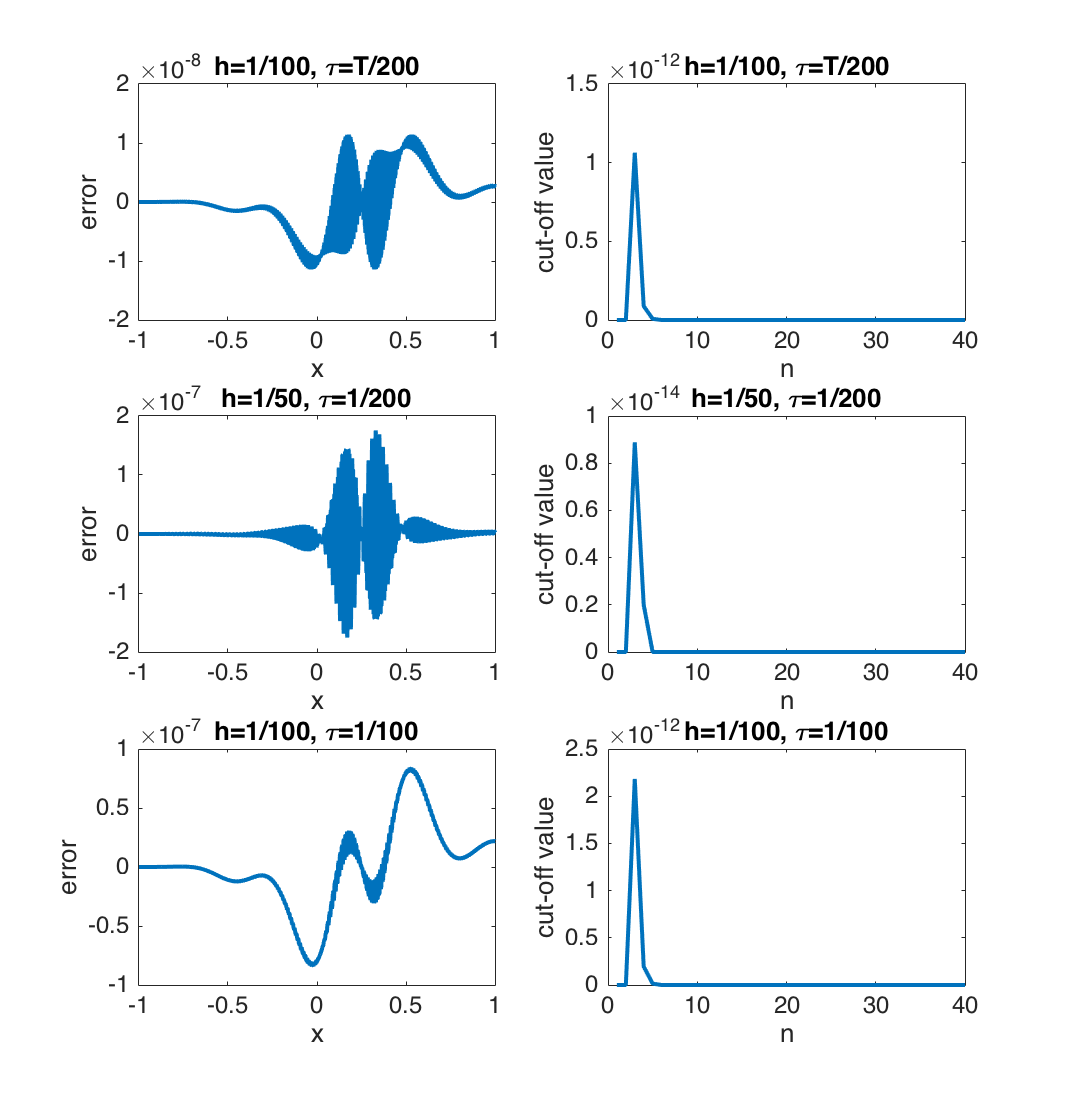}
\vspace{-15pt}
\caption{Error at $T=0.01$ and maximal cut-off value at each time level.
\label{fig:err_trun}}
%\vspace{-35pt}
\end{figure}

\begin{figure}[h]
  \centering
  \subfigure[$m = 1, \epsilon=0.1, T=2, \tau=1/150$]{\includegraphics[width=\textwidth]{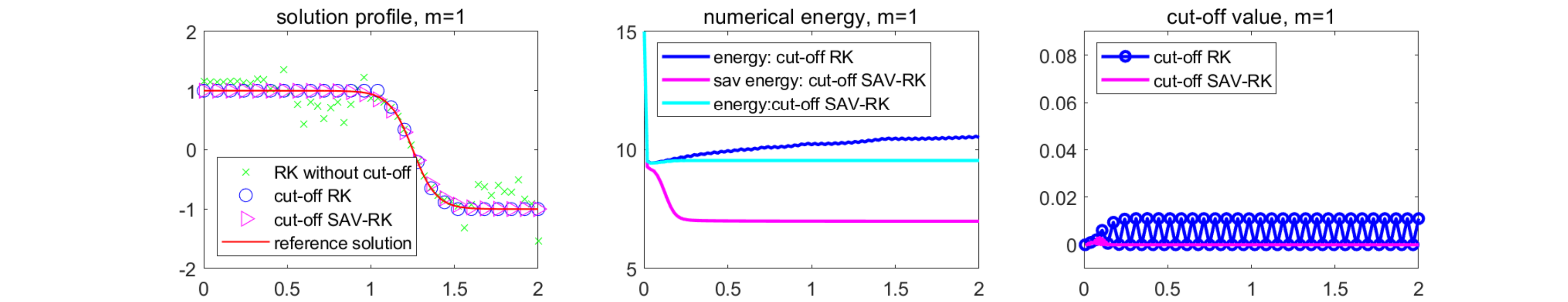}}
  \subfigure[$m = 2, \epsilon=0.1, T=2, \tau=1/250$]{\includegraphics[width=\textwidth]{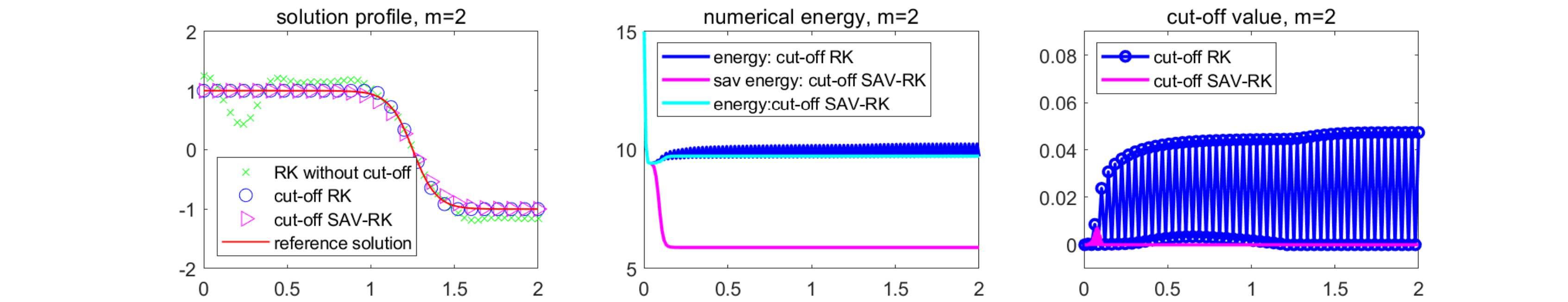}}
%  \subfigure[$m = 3, \epsilon=0.1, T=2, \tau=T/1000$]{\includegraphics[width=\textwidth]{1d_sol_energy_cut_3_1000.pdf}}
  \caption{Left: solution profiles of numerical solutions of RK,  cut-off RK and cut-off SAV-RK scheme.
  Middle: solution energy of cut-off RK and cut-off SAV-RK scheme. Right: cut-off values of cut-off RK and cut-off SAV-RK scheme.}
 \label{fig:energy}
\end{figure}

Finally, we test the numerical results in case of relatively large time steps, and compare the
numerical solutions of extrapolated RK, cut-off RK \eqref{eqn:RK-fully}-\eqref{eqn:RK-cut},
 and cut-off SAV-RK schemes \eqref{eqn:AC-sav-fully}-\eqref{eqn:truncation-fully-sav}, with $r=1$,
see Figure \ref{fig:energy}.
Without the cut-off postprocessing, the numerical solutions of RK scheme significantly
exceed the maximum bound,
and present oscillating solution profiles.
With the cut-off operation at each time step, the numerical solutions satisfy the maximum bound,
and present reasonable solution profiles. However,
numerical results show that the cut-off RK scheme
might produce a solution with a obviously increasing and oscillating
energy curve. This issue could be significantly improved by applying the cut-off SAV-RK method, whose solution
satisfy the maximum bound and the numerical energy is more stable.
Moreover, the numerical results show that the cut-off SAV-RK scheme will produce a more regular numerical solution
and smaller cut-off values, compared with the cut-off RK scheme.

\appendix

\section*{Appendix A}\label{sec:appendixA}
In this part, we sketch a proof for Lemma \ref{lem:RK-local-truncation}.
\begin{proof}
We note that the second relation in equation \eqref{eqn:RK-fully-trunc} implies
\begin{equation*}
 u(t_{ni}) -   u(t_{n-1}) -\tau\sum_{j=1}^m a_{ij}  u_t^{nj} =
 \tau \sum_{j=1}^m a_{ij} (\dot u_*^{nj} - u_t(t_{nj}))+
 \eta_{ni}\quad \text{for} ~~i=1,2,\dots,m.
\end{equation*}
Then we substitute the first relation of \eqref{eqn:RK-fully-trunc} and derive that
for $i=1,2,\dots,m$
\begin{equation*}
 u(t_{ni}) -   u(t_{n-1}) -\tau\sum_{j=1}^m a_{ij}  u_t^{nj} =
 \tau \sum_{j=1}^m a_{ij}   \Big(\sum_{\ell=1}^k L_{\ell}(t_{n-1}+c_j\tau) f(u(t_{n-\ell})) - f(t_{nj})\Big)+
 \eta_{ni}.
\end{equation*}
Define $\tilde\eta_{ni}$ as the left hand side of the above relation. %equation \eqref{eqn:err_splitted}.
Now we apply Taylor's expansion at $t_{n-1}$  and use \eqref{eqn:stage-order} to derive
\begin{equation*}
\begin{aligned}
\tilde\eta_{ni} =& \sum_{l=1}^{m}\frac{\tau^l}{(l-1)!}\l(\frac{c_i^l}l -
\sum_{j=1}^ma_{ij}c_j^{l-1}\r)u^{(\ell)}(t_n) +
\frac1{m!}\int_{t_{n-1}}^{t_{ni}} (t_{ni}-s)^{m} u^{(m+1)}(s) \d s \\
&+ \frac\tau{(m-1)!}\sum_{j=1}^ma_{ij}\int_{t_{n-1}}^{t_{nj}}
(t_{nj}-s)^{m-1} u^{(m+1)}(s) \d s\\
&=\frac1{m!}\int_{t_{n-1}}^{t_{ni}} (t_n-s)^{m} u^{(m+1)}(s) \d s
+ \frac\tau{(m-1)!}\sum_{j=1}^ma_{ij}\int_{t_{n-1}}^{t_{nj}}
(t_{nj}-s)^{m-1} u^{(m+1)}(s) \d s
\end{aligned}
\end{equation*}
Then we obtain the estimate for $\tilde\eta_{ni}$, with $i=1,2,\dots,m$, that
\begin{equation*}
\|\tilde\eta_{ni}\|_{H^1(\Omega)} \leq C\tau^{m+1} \|u^{(m+1)}\|_{C([t_{n-1},t_n];H^1(\Omega))}.
\end{equation*}
This together with the approximation property of Lagrange interpolation lead to
\begin{equation*}
    \|\eta_{ni}\|_{H^1(\Omega)} \leq C\Big(\tau^{k+1} \| f(u)  \|_{C^{k}([t_{n-k},t_n]; H^1\II)}
    + \tau^{m+1} \|u\|_{C^{(m+1)}([t_{n-1},t_n];H^1(\Omega))} \Big). %, \quad \text{for} ~~i=1,2,\dots,m.
\end{equation*}
for $i=1,2,\dots,m$. Similarly, we have
\begin{equation*}
    u(t_{n}) -  u(t_{n-1}) -\tau\sum_{i=1}^m b_i  u_t^{ni} =
    \tau \sum_{i=1}^m b_i  \Big(\sum_{\ell=1}^k L_{\ell}(t_{n-1}+c_i\tau) f(u(t_{n-\ell})) - f(t_{ni})\Big)  +
    \eta_{n}.
\end{equation*}
Take the left hand side as $\tilde{\eta_n}$. Then Taylor expansion and \eqref{eqn:accuracy} imply
\begin{equation*}
\begin{aligned}
\tilde\eta_{n} =\frac1{p!}\int_{t_{n-1}}^{t_n} (t_n-s)^p u^{(p+1)}(s) \d s
+ \frac\tau{(p-1)!}\sum_{i=1}^mb_i\int_{t_{n-1}}^{t_{ni}}
(t_{ni}-s)^{p-1} u^{(p+1)}(s) \d s.
\end{aligned}
\end{equation*}
This together with the approximation property of Lagrange interpolation leads to
\begin{equation*}
     \|\eta_{ni}\|_{H^1(\Omega)} \leq C\Big(\tau^{k+1} \| f(u)  \|_{C^{k}([t_{n-k},t_n]; H^1\II)}
    + \tau^{p+1} \|u \|_{C^{p+1}([t_{n-1},t_n];H^1(\Omega))} \Big).
\end{equation*}
Using the choice that $k=\min(p, m+1)$, we derive the desired result.
\end{proof}

\bibliographystyle{siam}

\end{document}